\def\DateTime{05/April/2017}
\def\Version{1.0}
\def\yes{\if00}
\def\no{\if01}
\def\iftenpt{\no}
\def\ifelevenpt{\no}
\def\iftwelvept{\yes}

\def\ifquery{\yes}
\def\sbullet{{\scriptscriptstyle\bullet}}

\iftenpt
\documentclass[leqno]{amsart}
\else\fi
\ifelevenpt
\documentclass[leqno,11pt]{amsart}
\else\fi
\iftwelvept
\documentclass[leqno,12pt]{amsart}
\else\fi

\usepackage{amssymb}
\usepackage{amscd}
\usepackage{verbatim,color}
\usepackage{mathrsfs,enumerate}
\usepackage{tikz}
\usepackage{tikz-cd}
\usepackage{manfnt}

\usepackage[sc]{mathpazo} 
\usepackage[T1]{fontenc}

\ifelevenpt
\else\fi

\iftwelvept
\else\fi

\theoremstyle{plain}
\newtheorem{Theorem}{Theorem}[section]
\newtheorem{Proposition}[Theorem]{Proposition}
\newtheorem{Lemma}[Theorem]{Lemma}

\newtheorem{Corollary}[Theorem]{Corollary}

\newtheorem{Claim}{Claim}[Theorem]

\theoremstyle{definition}
\newtheorem{Definition}[Theorem]{Definition}
\newtheorem{Remark}[Theorem]{Remark}
\newtheorem{DRemark}[Theorem]{{\scriptsize \textdbend}\,Remark}
\newtheorem{Example}[Theorem]{Example}

\numberwithin{equation}{section}
\def\rom{\textup}
\newcommand{\ZZ}{{\mathbb{Z}}}
\newcommand{\QQ}{{\mathbb{Q}}}
\newcommand{\RR}{{\mathbb{R}}}
\newcommand{\KK}{{\mathbb{K}}}
\newcommand{\CC}{{\mathbb{C}}}
\newcommand{\PP}{{\mathbb{P}}}

\newcommand{\OO}{{\mathcal{O}}}

\newcommand{\DDD}{{\mathscr{D}}}

\newcommand{\XXX}{{\mathscr{X}}}
\newcommand{\Proj}{\operatorname{Proj}}

\newcommand{\Pic}{\operatorname{Pic}}
\newcommand{\Rat}{\operatorname{Rat}}

\newcommand{\Hom}{\operatorname{Hom}}

\newcommand{\Ker}{\operatorname{Ker}}

\newcommand{\Spec}{\operatorname{Spec}}

\newcommand{\Supp}{\operatorname{Supp}}

\newcommand{\ord}{\operatorname{ord}}

\newcommand{\adeg}{\widehat{\operatorname{deg}}}
\newcommand{\zeros}{\operatorname{div}}

\newcommand{\Div}{\operatorname{Div}}
\newcommand{\aDiv}{\widehat{\operatorname{Div}}}

\newcommand{\vol}{\operatorname{vol}}
\newcommand{\avol}{\widehat{\operatorname{vol}}}

\newcommand{\an}{\operatorname{an}}

\newcommand{\PDiv}{\operatorname{PDiv}}

\newcommand{\pp}{{\mathfrak{p}}}

\ifquery
\def\query#1{\setlength\marginparwidth{65pt} %80pt for other size
\marginpar{\raggedright\fontsize{7.81}{9} 
\selectfont\upshape\hrule\smallskip 
#1\par\smallskip\hrule}} 

\else
\def\query#1{}
\fi
\newcommand{\ndot}{\raisebox{.4ex}{.}}

\definecolor{rose}{rgb}{1.0, 0.0, 0.5}
\definecolor{rosewood}{rgb}{0.4, 0.0, 0.04}
\definecolor{darkspringgreen}{rgb}{0.09, 0.45, 0.27}
\definecolor{orange-red}{rgb}{1.0, 0.27, 0.0}
\begin{document}

\title[Sufficient conditions for the Dirichlet property]
{Sufficient conditions for the Dirichlet property}% of arithmetic dynamic system}
\author{Huayi Chen}
\address{Universit\'{e} Paris Diderot - Paris 7,
Institut de Math\'{e}matiques de Jussieu - Paris Rive Gauche,
B\^{a}timent Sophie Germain,
Bo\^{i}te Courrier 7012,
75205 Paris Cedex 13, France
}
\email{huayi.chen@imj-prg.fr}
\author{Atsushi Moriwaki}
\address{Department of Mathematics, Faculty of Science,
Kyoto University, Kyoto, 606-8502, Japan}
\email{moriwaki@math.kyoto-u.ac.jp}
\date{\DateTime, (Version \Version)}
%\keywords{}
%\subjclass[2010]{Primary 14G40; Secondary 11G50, 37P30}
\begin{abstract}
The effectivity up to $\mathbb R$-linear equivalence (Dirichlet property) of pseudoeffective adelic $\mathbb R$-Cartier divisors is a subtle problem in arithmetic geometry. In this article, we propose sufficient conditions for the Dirichlet property by using the dynamic system in the classic Arakelov geometry setting. We also give a numerical criterion of the Dirichlet property for adelic $\mathbb R$-Cartier divisors on curves over a trivially valued field.
\end{abstract}

\maketitle

\section{Introduction}
Let $K$ be a number field and $X$ be a normal, projective and geometrically integral scheme over $\Spec K$. Recall that an \emph{adelic $\mathbb R$-Cartier divisor} on $X$ is by definition a couple $\overline D=(D,g)$, where $D$ is an $\mathbb R$-Cartier divisor on $X$ and $g=(g_v)_{v\in M_K}$ is a family of Green functions indexed by the set $M_K$ of all places of $K$. Note that the adelic $\mathbb R$-Cartier divisors on $X$ form a vector space $\widehat{\mathrm{Div}}_{\mathbb R}(X)$ over $\mathbb R$ and one has a natural $\mathbb R$-linear homomorphism from $\mathrm{Rat}(X)_{\mathbb R}^{\times}:=\mathrm{Rat}(X)^{\times}\otimes_{\mathbb Z}\mathbb R$ to $\widehat{\mathrm{Div}}_{\mathbb R}(X)$, where $\mathrm{Rat}(X)$ denotes the field of rational functions on $X$. Two adelic $\mathbb R$-Cartier divisors are said to be \emph{$\mathbb R$-linearly equivalent} if their difference is $\RR$-principal, that is, lies in the image of the canonical map $\mathrm{Rat}(X)_{\mathbb R}^{\times}\rightarrow\widehat{\mathrm{Div}}_{\mathbb R}(X)$. Several positivity conditions in algebraic geometry (in particular bigness and 
pseudo-effectivity
 of $\mathbb R$-Cartier divisors) have the counterpart in the setting of adelic $\mathbb R$-Cartier divisors. We refer the readers to \cite{Moriwaki16} for more details.

In \cite{MoD}, Moriwaki has compared the pseudo-effectivity and the effectivity up to $\mathbb R$-linear equivalence in the setting of adelic $\mathbb R$-Cartier divisors in introducing the so-called \emph{Dirichlet property}. We say that an adelic $\mathbb R$-Cartier divisor satisfies the Dirichlet property if it is $\mathbb R$-linearly equivalent to an effective adelic $\mathbb R$-Cartier divisor. It can be shown that the Dirichlet property implies the pseudo-effectivity. It is then quite natural to ask if all pseudo-effective $\mathbb R$-Cartier divisors satisfy the Dirichlet property, and if it is not the case, how to determine the Dirichlet property of pseudo-effective divisors.

It is worth to mention that the first question above has a confirmative answer if 
$X = \Spec K$. 
It can be deduced from the Dirichlet 
unit
theorem for number fields, from which comes the terminology of Dirichlet property, see \cite{MoD} for more details. Moreover, it has been shown in the same reference that a pseudo-effective adelic $\mathbb R$-Cartier divisor $(D,g)$ with $D$ principal also satisfies the Dirichlet property. However, in higher dimensional case the first question above has a negative answer. We refer the readers to \cite{DsysDirichlet} for counterexamples and obstructions to the Dirichlet property. 

The problem of Dirichlet property can be stated in the classic algebraic geometry setting and other arithmetic settings such as arithmetic varieties over a function field or a field with trivial valuation. The purpose of this article is to provide some sufficient conditions for the Dirichlet property in divers situations, with highlights on the role of finiteness conditions.
We first focus on the number field case, using the dynamic system on arithmetic varieties. Let $X$ be a normal, projective and geometrically integral scheme over a number field $K$. Let $f : X \to X$ be a surjective endomorphism over $K$.
Let $D$ be an effective and ample $\RR$-Cartier divisor on $X$. We assume that
there are $d \in \RR_{>1}$ and $\varphi \in \Rat(X)^{\times}_{\RR}$ such that
$f^*(D) = d D + (\varphi)$.
By \cite[Section~3]{DsysDirichlet},
there is a unique family $g = \{ g_v \}_{v \in M_K}$ of $D$-Green functions of $C^0$-type
with $f^*(D, g) = d(D, g) + \widehat{(\varphi)}$.
The pair $\overline{D} = (D, g)$ is called the \emph{canonical
compactification} of $D$. 
Note that $\overline{D}$ is nef in the arithmetical sense (for details, see \cite[Lemma~4.1]{DsysDirichlet}).
We establish the following sufficient condition for the Dirichlet property.  

\begin{Theorem}\label{Thm:dynamic example}Suppose that there exists a finite dimensional vector subspace of $\mathrm{Rat}(X)_{\mathbb R}^{\times}$ which contains $\varphi$ and is stable by $f^*$.
Then $\overline{D}$ satisfies the Dirichlet property. In particular, if $f^*(D) = dD$,
then $\overline{D}$ satisfies the Dirichlet property.
\end{Theorem}

We can for instance apply the above theorem to the case where $X = \PP^n_K = \Proj(K[T_0, T_1, \ldots, T_n])$,
$D = \{ T_0 = 0\}$ and $f$ is a polynomial map, that is,
\[
f\left(\PP^n_K \setminus \{ T_0 = 0 \}\right) \subseteq \PP^n_K \setminus \{ T_0 = 0 \}.
\]
If $f$ is not an automorphism, then $f^*(D) = dD$ for some $d \in \ZZ_{\geqslant 2}$.
Therefore, the above theorem implies the effectivity of $\overline{D}$.
For example, if $f$ is given by $f(T_0:T_1) = (T_0^2 : T_1^2 + cT_0^2)$ ($c \in K$) on $\PP^1_K$,
then $\overline{D}$ is effective.
Even if the Julia set $J(f_{v})$ of $f_{v}$ ($v \in M_K$) is complicated,
$|1|_{g_{v}} = 1$ on $J(f_{v})$
by \cite[Lemma~2.1 and Remark~2.3]{DsysDirichlet}. More concrete examples of adelic $\mathbb R$-Cartier divisors verifying the sufficient condition will be discussed in 
Example~\ref{Example:Dirichlet}.

We then consider the Dirichlet property in the setting of arithmetic varieties over a trivially valued field. We consider an integral projective scheme $X$ over a field $K$. We equip $K$ with the trivial absolute value $|\ndot|$ (namely $|a|=1$ for any $a\in K^{\times}$). Denote by $X^{\mathrm{an}}$ the Berkovich  space associated with $X$. If $D$ is an $\mathbb R$-Cartier divisor on $X$, by \emph{$D$-Green function of $C^0$-type}, or \emph{Green function} of $D$, we refer to a continuous function on the complementary of the analytification of the support of $D$, which is locally of the form $\varphi-\log|f|$, where $\varphi$ is a continuous function on $X^{\mathrm{an}}$ and $f$ is an element of $\mathrm{Rat}(X)^{\times}_{\mathbb R}$ which defines $D$ locally. The pair $\overline D=(D,g)$ is called an \emph{adelic $\mathbb R$-Cartier divisor on $X$}. The analogue of the arithmetic volume function can be defined in this setting, and the bigness and  pseudo-effectivity of adelic $\mathbb R$-Cartier divisors are defined in a similar way as in the classic arithmetic framework. It is then a natural question to determine sufficient conditions for a pseudo-effective adelic $\mathbb R$-Cartier divisor to be $\mathbb R$-linearly equivalent to an effective one. Even for this simple setting where only the trivial valuation is considered in the adelic structure, this problem  still seems to be very subtle. However, in the case where $X$ is a regular curve over $\Spec K$ such that
$\dim_{\QQ} (\Pic(X) \otimes \QQ) = 1$, 
we have a complete answer to this problem. In fact, when $X$ is a regular curve over $\Spec K$, the Berkovich  space $X^{\mathrm{an}}$ can be illustrated by an infinite tree of depth $1$
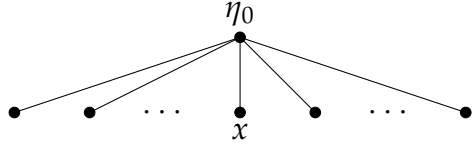
\vspace{3mm}
\begin{center}
\begin{tikzpicture}
\filldraw(0,1) circle (2pt) node[align=center, above]{$\eta_0$};
\filldraw(-3,0) circle (2pt) ;
\draw (-1,0) node{$\cdots$};
\filldraw(-2,0) circle (2pt) ;
\filldraw(-0,0) circle (2pt) node[align=center, below]{$x$} ;
\filldraw(1,0) circle (2pt) ;
\draw (2,0) node{$\cdots$};
\filldraw(3,0) circle (2pt) ;
\draw (0,1) -- (0,0);
\draw (0,1) -- (-3,0);
\draw (0,1) -- (1,0);
\draw (0,1) -- (-2,0);
\draw (0,1) -- (3,0);
\end{tikzpicture}
\end{center}
\vspace{3mm}
where the root vertex $\eta_0$ corresponds to the generic point and the trivial absolute value on the field $\mathrm{Rat}(X)$, and the leaves are parametrised by the set $X^{(1)}$ of closed point in $X$ (together with the trivial absolute value on the corresponding residue field). We denote by $i:X\rightarrow X^{\mathrm{an}}$ the map sending the generic point $\eta$ of $X$ to $\eta_0$ and each closed point $x$ to the corresponding leaf in the tree. Each branch $[\eta_0,x]$ with $x\in X^{(1)}$ is parametrised by $t:[\eta_0,x]\rightarrow[0,+\infty]$. Any $\xi\in 
[\eta_0,x[$ corresponds to the generic point of $X$ and the field $\mathrm{Rat}(X)$ equipped with the absolute value $|\ndot|_\xi$ such that \[|\ndot|_\xi=\mathrm{e}^{-t(\xi)\mathrm{ord}_x(\ndot)}\;\text{on $\mathrm{Rat}(X)^{\times}$}\]
where $\mathrm{ord}_x(\ndot)$ is the discrete valuation on $\mathrm{Rat}(X)$ with valuation ring $\mathcal O_{X,x}$. Moreover, $t(x)=+\infty$. The space $X^{\mathrm{an}}$ is equipped with the Berkovich topology, whose restriction on each branch $[\eta_0,x]$ corresponding to the usual topology on $[0,+\infty]$ via the parametrisation $t(\ndot)$, and any open neighbourhood of $\eta_0$ contains all but finitely many branches.

Given a continuous function $g$ on $X^{\mathrm{an}}\setminus i(X^{(1)})$, we define a family of invariants
\[\forall\,x\in X^{(1)},\quad \mu_x(g):=\inf_{\xi\in\,]\eta_0,x[}\frac{g(\xi)}{t(\xi)}\in \mathbb R\cup\{-\infty\}.\]
We establish the following result (see Theorems \ref{Pro:Dirichlet}, \ref{Pro: Dirichlet property and mu tot}, and Remark \ref{Rem: obstruction}).
\begin{Theorem}
Let $X$ be a regular projective curve over $\Spec K$ such that 
$\dim_{\QQ} (\Pic(X) \otimes \QQ) = 1$.
Let $\overline D=(D,g)$ be an adelic $\mathbb R$-Cartier divisor on $X$ such that $D$ is big. Then for all but a finite number of $x\in X^{(1)}$, one has $\mu_x(g)\leqslant 0$. Moreover, with the notation
\[ \mu_{\mathrm{tot}}(g):=\sum_{x\in X^{(1)}}\mu_x(g)[\kappa(x):K],
\]
where $\kappa(x)$ denotes the residue field of $x$, the following statements hold.
\begin{enumerate}[(1)]
\item The adelic $\mathbb R$-Cartier divisor $\overline D$ is pseudo-effective if and only if $\mu_{\mathrm{tot}}(g)\geqslant 0$.
\item The adelic $\mathbb R$-Cartier divisor $\overline D$ satisfies the Dirichlet property if and only if $\mu_x(g)\geqslant 0$ for all but a finite number of $x\in X^{(1)}$, and $\mu_{\mathrm{tot}}(g)\geqslant 0$.
\end{enumerate}
\end{Theorem}

Let $f : \PP^1_K \to \PP^1_K$ be an endomorphism and $D$ be an $\RR$-Cartier divisor on $\PP^1_K$
such that $\deg(D) \geqslant 0$ and
$f^*(D) = dD + (\varphi)$ for some $d \in \RR_{>1}$ and $\varphi \in \Rat(X)^{\times}_{\RR}$.
Let $g$ be a unique $D$-Green function of $C^0$-type with $f^*(g) = dg - \log |\varphi|$.
As a corollary of the above theorem, we can conclude that 
$\overline{D} = (D, g)$ satisfies the Dirichlet
property (cf. Corollary~\ref{cor:Dirichlet:endomorphism:curve},
Proposition~\ref{prop:semiample:plurisubharmonic:endomorphism}, and Corollary~\ref{cor:plurisubharmonic:pseudoeffective:Dirichlet}). 
This is an essentially different point from the classic setting.

Surprisingly, we observe again a finiteness condition in the comparison of the pseudo-effectivity and the Dirichlet property. These results suggest that the functional obstructions to the Dirichlet property introduced in \cite{DsysDirichlet} may not be the only obstruction.  

The rest of the article is organised as follows. We first introduce the notation and conventions that will be used throughout the article.
In the second section we recall some basic constructions on Berkovich spaces such as Green functions of ($\mathbb R$-)Cartier divisors and prove preliminary  results which are useful for the proof of the main theorems. In the third section, 
we prove Theorem \ref{Thm:dynamic example} and provide several concrete applications of the theorem. 
In the fourth section, 
we introduce the framework of Arakelov geometry over a trivially valued field and discuss several positivity conditions such as bigness and pseudo-effectivity in this framework. 
Finally in the fifth section 
we discuss the Dirichlet property in the setting of Arakelov geometry over a trivially valued field.

\renewcommand{\thesubsubsection}{\arabic{subsubsection}}

\subsection*{Conventions and terminology}
Throughout this subsection,
let $\KK$ be either $\QQ$ or $\RR$.

\subsubsection{}
\label{CT:tensor:R}
Let $(G,\cdot)$ be a multiplicative abelian group. 
The tensor product $G \otimes_{\ZZ} \KK$ is denoted by $G_{\KK}$.
For $\phi_1, \ldots, \phi_r \in G_{\KK}$ and
$A = (a_1, \ldots, a_r) \in \KK^r$, we set $\phi^A := \phi_1^{a_1} \cdots \phi_r^{a_r}$
in $G_{\KK}$
for sake of simplicity.

\subsubsection{}
\label{CT:R:Cartier:div}

Let $X$ be a Noetherian integral scheme and $\mathscr M_X$ be the sheaf of
rational functions on $X$.
We define $\Div(X)$ and $\Div_{\KK}(X)$ to be
\[
\Div(X) := H^0(X, \mathscr M_X^{\times}/\OO_X^{\times})\ \text{and}\ 
\Div_{\KK}(X) := H^0(X, \mathscr M_X^{\times}/\OO_X^{\times}) \otimes \KK,
\]
whose elements are called {\em Cartier divisors} and {\em $\KK$-Cartier divisors} on $X$, respectively.
Let $\Rat(X)$ be the field consisting of all rational functions on $X$
and $\PDiv(X)$ be the subgroup of $\Div(X)$ consisting principal divisors on $X$, that is, $\PDiv(X) := \{ (\phi) \mid \phi \in \Rat(X)^{\times}\}$.
We call any element of $\Rat(X)^{\times}_{\KK}$ a \emph{${\KK}$-rational function} on $X$.
Note that the natural homomorphism $\Div(X) \to \Div_{\KK}(X)$
is not necessarily injective (see Remark~\ref{DRem:R:Cartier:vs:Cartier}).
A $\KK$-Cartier divisor $D$ on $X$ is locally given by $f \in \Rat(X)^{\times}_{\KK}$, which is called
a \emph{local equation} of $D$.
For a $\KK$-rational function $\varphi$ on $X$,
we can define a $\KK$-Cartier divisor $(\varphi)$ on $X$ by
considering the local equation $\varphi$ everywhere.
The $\KK$-Cartier divisor $(\varphi)$ is called a {\em $\KK$-principal divisor} of $\varphi$.
We denote the vector subspace of $\Div_{\KK}(X)$ consisting of $\KK$-principal divisors on $X$ by
$\PDiv_{\KK}(X)$. 
Note that $\PDiv_{\KK}(X) = \PDiv(X) \otimes_{\mathbb Z} \KK$.
Moreover, $\Pic(X) = \Div(X)/\PDiv(X)$ and $\Pic_{\KK}(X) := \Pic(X) \otimes_{\mathbb Z} \KK \cong \Div_{\KK}(X)/\PDiv_{\KK}(X)$.

A Cartier divisor is said to be {\em effective} if every local equation can be taken as a regular function.
Furthermore, a $\KK$-divisor $D$ on $X$ is said to be {\em $\KK$-effective}, denoted by $D \geqslant_{\KK} 0$,
if there are effective Cartier divisors $D_1, \ldots, D_r$ on $X$ and
$a_1, \ldots, a_r \in \KK_{>0}$ such that $ D = a_1D_1 + \cdots + a_r D_r$.
For a $\QQ$-Cartier divisor $D$, by applying (2) in Claim~\ref{claim:prop:Q:PSH:equiv:R:PSH:01} to
the case where $V = \Rat(X)^{\times}_{\QQ}$, $x = \text{a local equation of $D$}$ and
$x_i = \text{a local equation of $D_i$}$ ($i=1, \ldots, r$), we can see that
$D \geqslant_{\QQ} 0$ if and only if $D \geqslant_{\RR} 0$. However, 
a Cartier divisor which is effective as a $\KK$-Cartier divisor is not
necessarily effective as a Cartier divisor (see Remark~\ref{DRem:R:Cartier:vs:Cartier}).
For simplicity, a $\KK$-effective $\KK$-Cartier divisor is often said to be
effective and the notation $D \geqslant_{\KK} 0$ is denoted by $D \geqslant 0$.

Let $D$ be a $\KK $-Cartier divisor on $X$ and, for each $x \in X$, let $f_x$ ($\in \Rat(X)^{\times}_{\KK }$) be a local equation of $D$ at $x$.
We define $\Supp_{\KK }(D)$ to be
\[
\Supp_{\KK }(D) := \{ x \in X \mid f_x \not\in (\OO_{X,x}^{\times})_{\KK } \}.
\]
Note that the above definition does not depend on the choice of $f_x$
because if $f'_x$ is another local equation of $D$ at $x$, then $f_x/f'_x \in (\OO_{X, x}^{\times})_{\KK }$.
Moreover, $\Supp_{\KK }(D)$ is closed by \cite[Proposition~1.2.1]{Moriwaki16}.
In addition, for a $\QQ$-Cartier divisor $D$, $\Supp_{\QQ}(D) = \Supp_{\RR}(D)$ by 
(1) in Claim~\ref{claim:prop:Q:PSH:equiv:R:PSH:01}. If $D$ is a Cartier divisor, then we can take $f_x$ belonging to $\Rat(X)^{\times}$, so that
we can define another $\Supp_{\ZZ}(D)$ to be
\[
\Supp_{\ZZ}(D) := \{ x \in X \mid f_x \not\in \OO_{X,x}^{\times} \}.
\]
Obviously $\Supp_{\KK }(D) \subseteq \Supp_{\ZZ}(D)$, but $\Supp_{\KK }(D) \not= \Supp_{\ZZ}(D)$ in general
(for details, see \cite[Subsection~1.2]{Moriwaki16} or Remark~\ref{DRem:R:Cartier:vs:Cartier}).
For sake of simplicity, we often denote $\Supp_{\KK }(D)$ by $\Supp(D)$.
Furthermore, for a $\KK $-Cartier divisor $D$ on $X$, we define $H^0(X, D)$ to be
\[
H^0(X, D) := \{ \phi \in \Rat(X)^{\times} \mid D + (\phi) \geqslant_{\KK } 0 \} \cup \{ 0 \}.
\]
We assume that $D$ is a Cartier divisor on $X$.
Let $\OO_X(D)$ be an invertible sheaf associated with $D$.
Then we have a canonical injective homomorphism 
\[
H^0(X, \OO_X(D)) \to H^0(X, D).\]
Note that it is not necessarily surjective (for details, see Remark~\ref{DRem:R:Cartier:vs:Cartier}).

\subsubsection{}
\label{CT:semiample}
Let $X$ be an integral projective scheme over a field $K$ and
$D$ be either a Cartier divisor or a $\QQ$-Cartier divisor or an $\RR$-Cartier divisor on $X$.
We say $D$ is {\em semiample} if one of the following conditions is satisfied according to the class of $D$:

\medskip
$\bullet$ {\bf Cartier divisor}:
there is a positive integer $n$ such that
$\OO_X(nD)$ is generated by global sections.

$\bullet$ {\bf $\QQ$-Cartier divisor}:
there is a positive integer $m$ such that
$mD$ can be represented by a semiample Cartier divisor, that is,
there is a semiample Cartier divisor $A$ on $X$ such that
$mD$ is the image of $A$ via the natural homomorphism
$\Div(X) \to \Div_{\QQ}(X)$.

$\bullet$ {\bf $\RR$-Cartier divisor}:
there are semiample Cartier divisors
$A_1, \ldots, A_r$ and non-negative real numbers $a_1, \ldots, a_r$
such that $D = a_1 A_1 + \cdots + a_r A_r$.

\medskip\noindent
Note that every $\RR$-principal divisor is semiample because
every principal divisor is semiample.

\subsubsection{}
\label{CT:number:field:places}
Let $K$ be a number field, that is, $K$ is a finite extension field over $\QQ$.
Let $O_K$ be the ring of integers in $K$.
We set  $M_K^{\mathrm{fin}} := \Spec(O_K) \setminus \{ (0) \}$, 
which is referred as the set of finite places of $K$.
Moreover, the set of all embeddings $K \hookrightarrow \CC$ is denoted by $K(\CC)$.
By abuse of notation, $K(\CC)$ is referred to as the set of infinite places of $K$ and
it is often denoted by $M_K^{\infty}$. 
We set $M_K := M_K^{\mathrm{fin}} \cup M_K^{\infty}$. 
Note that $M_K$ is slightly different from the notation in \cite{DsysDirichlet}.
Let $X$ be a normal, projective and geometrically integral scheme over $\Spec K$.
For each $v \in M_K$, $K_v$ $X_v$ and $X_v^{\mathrm{an}}$ are defined as follows (see also \S\ref{Reminder Berkovich}):
{\allowdisplaybreaks
\begin{align*}
& \text{\bf $\bullet$ Case $v = \mathfrak{p} \in M_K^{\mathrm{fin}}$} : \\
& \hskip3em \begin{cases}
K_v := \text{the completion of $K$ at $v$}, \\
X_v := X \times_{\Spec(K)} \Spec(K_v), \\
X_v^{\mathrm{an}} := \text{the analytification of $X_v$ in the sense of Berkovich}.
\end{cases}\\[1ex]
& \text{\bf $\bullet$ Case $v = \sigma \in M_K^{\infty}$} : \\
& \hskip3em \begin{cases}
K_v := K \otimes_{K}^{\sigma} \CC \ \text{with respect to $v : K \hookrightarrow \CC$}, \\
X_v := X \times^{\sigma} _{\Spec(K)} \Spec(\CC) \ \text{with respect to $v : K \hookrightarrow \CC$}, \\
X_v^{\mathrm{an}} := X_v(\CC).
\end{cases}
\end{align*}}

\subsubsection{}
\label{CT:adelic:R:Cartier:div}
Let $K$ be a number field and $X$ be a normal, projective and geometrically integral scheme over $\Spec K$.
A pair $\overline{D} = (D, g)$ of an $\RR$-Cartier divisor $D$ on $X$ and a collection
\[
g = \left\{ g_\pp \right\}_{\pp \in M_K} 
\cup \left\{ g_{\sigma} \right\}_{\sigma \in M^{\infty}_K}
\]
of $D$-Green functions of $C^0$-type
is called \emph{an adelic arithmetic $\RR$-Cartier divisor of $C^0$-type on $X$}
if the following conditions are satisfied:
\begin{enumerate}
\item
For each $\pp \in M_K^{\mathrm{fin}}$,
$g_\pp$ is a $D$-Green function of $C^0$-type on $X^{\mathrm{an}}_\pp$.
In addition, there are a non-empty open set $U$ of $\Spec(O_K)$,
a model $\XXX_U$ of $X$ over $U$ and an $\RR$-Cartier divisor $\DDD_U$ on $\XXX_U$
such that $\DDD_U \cap X = D$ and $g_\pp$ is a $D$-Green function induced by the model
$(\XXX_U, \DDD_U)$ for all $\pp \in U \cap M_K$. 

\item
For each $\sigma \in M_K^{\infty}$, $g_{\sigma}$ is a $D$-Green function of $C^0$-type on $X^{\mathrm{an}}_{\sigma}$.
Moreover, the function $\left\{ g_{\sigma} \right\}_{\sigma \in M^{\infty}_K}$ is an $F_{\infty}$-invariant,
that is, for all $\sigma \in M_K^{\infty}$, $g_{\overline{\sigma}} \circ F_{\infty} = g_{\sigma}$,
where $F_{\infty} : X^{\mathrm{an}}_{\sigma} \to X^{\mathrm{an}}_{\overline{\sigma}}$ is an anti-holomorphic map induced by
the complex conjugation.
\end{enumerate}
The space of all adelic arithmetic
$\RR$-Cartier divisors of $C^0$-type on $X$
is denoted by 
$\aDiv_{\RR}(X)$.
For an adelic arithmetic $\RR$-Cartier divisor $\overline{D}$ of $C^0$-type on $X$,
we define $\hat{\Gamma}(X, \overline{D})^{\times}_{\RR}$ to be
\[
\hat{\Gamma}(X, \overline{D})^{\times}_{\RR} := \{ s \in \Rat(X)^{\times}_{\RR} \mid
\overline{D} + \widehat{(s)} \geqslant 0 \}.
\]
We say $\overline{D}$ \emph{satisfies the Dirichlet property} if 
$\hat{\Gamma}(X, \overline{D})^{\times}_{\RR} \not= \emptyset$.

\section{Green functions on Berkovich analytic spaces}
Let $K$ be a field and $|\ndot|$ be a complete absolute value of $K$.
The absolute value $|\ndot|$ might be trivial (so that $|a|=1$ for any $a\in K\setminus\{0\}$).
Let $X$ be an integral projective scheme over $\Spec K$
and $X^{\mathrm{an}}$ be the analytification of $X$ in the sense of Berkovich.
In this section, we consider a Green function on $X^{\mathrm{an}}$ associated with an $\RR$-Cartier divisor. 

\subsection{Reminder on Berkovich spaces}\label{Reminder Berkovich}

Let $X$ be a scheme over $\Spec K$. As a set, $X$ identifies with the colimit of the functor $F_X$, from the category $\mathbf{E}_K$ of fields extensions of $K$ and $K$-linear field homomorphisms, to the category $\mathbf{Set}$ of sets, which sends any extension $K'/K$ to the set of $K$-morphisms from $\Spec K'$ to $X$. The Berkovich space (see \cite{Berkovich90}) $X^{\mathrm{an}}$ associated with $X$ can also be defined in a similar way. We denote by $\mathbf{VE}_K$ the category of valued extensions of $K$ and $K$-linear homomorphisms preserving absolute values. More precisely, any objet of $\mathbf{VE}_K$ is of the form $(K',|\ndot|')$, where $K'$ is an extension of $K$ and $|\ndot|'$ is an absolute value on $K'$ extending $|\ndot|$. We let $\varpi:\mathbf{VE}_{K}\rightarrow\mathbf{E}_K$ be the forgetful functor sending $(K',|\ndot|')$ to $K'$. As a set the Berkovich space $X^{\mathrm{an}}$ is then defined as the colimit of the composed functor $F_X\circ\varpi$. By the universal property of colimit one has a natural map $j:X^{\mathrm{an}}\rightarrow X$, called the \emph{specification map}. This construction is functorial: for any morphism of $K$-schemes $\varphi:X\rightarrow Y$, the universal property of colimit determines a map $\varphi^{\mathrm{an}}:X^{\mathrm{an}}\rightarrow Y^{\mathrm{an}}$.

Let $\xi$ be a point of $X^{\mathrm{an}}$ and $\kappa(\xi)$ be the residue field of $j(\xi)\in X$, called the \emph{residue field of $\xi$}. If $y:\Spec K'\rightarrow X$ is a $K$-morphism, where $(K',|\ndot|_y)$ is a valued extension of $(K,|\ndot|)$, which represents the point $\xi\in X^{\mathrm{an}}$, then the morphism factorises through the canonical $K$-morphism $\Spec\kappa(\xi)\rightarrow X$ and the restriction of $|\ndot|_y$ on $\kappa(\xi)$ does not depend on the representative $y$ of the class $\xi$. We denote by $|\ndot|_{\xi}$ this absolute value. We emphasis that two different points $\xi$ and $\xi'$ of $X^{\mathrm{an}}$ may have the same residue field. However, in this case $|\ndot|_{\xi}$ and $|\ndot|_{\xi'}$ are different.  

On the Berkovich space $X^{\mathrm{an}}$ there is a natural topology which is the most coarse topology making the specification map $j:X^{\mathrm{an}}\rightarrow X$ continuous, where we consider the Zariski topology on $X$. This topology is called the \emph{Zariski topology} on $X^{\mathrm{an}}$. Berkovich has introduced a finer topology as follows, called \emph{Berkovich topology} nowadays. Let $f$ be a regular function on a Zariski open subset $U$ of $X$. Recall that $f$ corresponds to a morphism from $U$ to $\mathbb A_K^1$. Therefore, for any $\xi\in U^{\mathrm{an}}$, the morphism $f$ determines an element $f(\xi)$ in $\kappa(\xi)$. We denote by $|f|(\xi)$ the absolue value $|f(\xi)|_\xi$. The \emph{Berkovich topology} on $X^{\mathrm{an}}$ is defined as the most coarse topology on $X^{\mathrm{an}}$ which makes the specification map $j$ and all functions of the form $|f|$ continuous, where $f$ runs over all regular functions on Zariski open subsets of $X$ (see \cite[\S3.4]{Berkovich90} for more details). If $f:X\rightarrow Y$ is a morphism of $K$-schemes, then the map $f^{\mathrm{an}}:X^{\mathrm{an}}\rightarrow Y^{\mathrm{an}}$ is continuous with respect to the Berkovich topology.

\subsection{Green function}
Let $X$ be an integral projective scheme over $\Spec K$. We denote by $\widehat{C}^0(X^{\mathrm{an}})$ the set of continuous functions on a non-empty Zariski open subset of $X^{\mathrm{an}}$, modulo the following equivalence relation
\[f\sim g \;\Longleftrightarrow\;\text{$f$ and $g$ coincide on a non-empty Zariski open subset}.\]
Note that the addition and the multiplication of functions induce a structure of $\mathbb R$-algebra on $\widehat{C}^0(X^{\mathrm{an}})$. Moreover, for any non-empty Zariski open subset $U$ of $X$, we have a natural homomorphism of $\mathbb R$-algebras from $C^0(U^{\mathrm{an}})$ to $\widehat{C}^0(X^{\mathrm{an}})$. Since $U^{\mathrm{an}}$ is dense in $X^{\mathrm{an}}$ (see \cite[Corollary 3.4.5]{Berkovich90}), this homomorphism is injective. Therefore, by abuse of notation we may consider any function in $C^0(U^{\mathrm{an}})$ as an element in $\widehat{C}^0(X^{\mathrm{an}})$. We say that an element of $\widehat{C}^0(X^{\mathrm{an}})$ \emph{extends to a continuous function on $U^{\mathrm{an}}$} if it belongs to the image of the canonical homomorphism $C^0(U^{\mathrm{an}})\rightarrow\widehat{C}^0(X^{\mathrm{an}})$.

\begin{Example}
Let $X$ be an integral projective scheme over $\Spec K$. If $f$ is a non-zero rational function on $X$, then it coincides with an invertible regular function on some non-empty Zariski open subset $U$ of $X$. Therefore $\log|f|$ determines an element of $\widehat{C}^0(X^{\mathrm{an}})$, which does not depend on the choice of $U$. The map from $\mathrm{Rat}(X)^{\times}$ to $\widehat{C}^0(X^{\mathrm{an}})$ sending $f\in \mathrm{Rat}(X)^{\times}$ to $\log|f|$ is a group homomorphism, and hence induces an $\mathbb R$-linear map $\mathrm{Rat}(X)^{\times}_{\mathbb R}\rightarrow\widehat{C}^0(X)$ which we still denote by $\log|\ndot|$.
\end{Example}

\begin{Definition}\label{Def:Green functions}
Let $D$ be a Cartier divisor on $X$. We call \emph{$D$-Green function of $C^0$-type} (or simply \emph{Green function of $D$})
any element $g\in\widehat{C}^0(X^{\mathrm{an}})$ such that, for any element $f\in\mathrm{Rat}(X)^{\times}$ which defines the Cartier divisor $D$ locally on a non-empty Zariski open subset $U$ of $X$, the element $g+\log|f|$ of $\widehat{C}^0(X^{\mathrm{an}})$ extends to a continuous function on $U^{\mathrm{an}}$.

Similarly if $\mathbb K=\mathbb Q$ or $\mathbb R$ and if $D$ is a $\mathbb K$-Cartier divisor on $X$,
we call 
\emph{$D$-Green function of $C^0$-type} or \emph{Green function of $D$}
any element $g\in\widehat{C}^0(X^{\mathrm{an}})$ such that, for any element $f\in\mathrm{Rat}(X)^{\times}_{{\mathbb K}}$ which defines the ${\mathbb K}$-Cartier divisor $D$ locally on a non-empty Zariski open subset $U$ of $X$, the element $g+\log|f|$ of $\widehat{C}^0(X^{\mathrm{an}})$ extends to a continuous function on $U^{\mathrm{an}}$. 

For ${\mathbb K}$-Cartier divisors $D$ and $D'$ and $a, a' \in {\mathbb K}$, 
it is easy to see that if $g$ and $g'$ are Green functions of $D$ and $D'$, respectively,
then $ag + a'g'$ is a Green function of $aD + a'D$.
In particular, if $g$ is a Green function of the trivial Cartier divisor or the trivial $\mathbb K$-Cartier divisor, then it extends to a continuous function on $X^{\mathrm{an}}$. 

Let $\{g_n\}_{n=1}^\infty$ be a sequence of $D$-Green functions of $C^0$-type and $g$ be a $D$-Green function of $\mathcal C^0$-type. Let $\theta_n$ be the unique continuous extension of $g-g_n$ on $X^{\mathrm{an}}$. We say that the sequence $\{g_n\}_{n=1}^\infty$ \emph{converges uniformly} to $g$ if $\lim_{n\rightarrow\infty}\|\theta_n\|_{\sup}=0$.
\end{Definition}

\begin{Example}
Let $f$ be an element in $\mathrm{Rat}(X)^{\times}_{\mathbb R}$ and $(f)$ be the $\mathbb R$-Cartier divisor on $X$ defined by $f$. Then the element $-\log|f|\in\widehat{C}^0(X^{\mathrm{an}})$ is a Green function of $(f)$. 
\end{Example}

\begin{Remark}\label{Rem: continuous metric}

Green functions are closely related to continuous metrics on line bundles. Let $L$ be an invertible $\mathcal O_X$-module. By \emph{continuous metric} on $L$, we refer to a family $\phi=(|\ndot|_\phi(x))_{x\in X^{\mathrm{an}}}$, where for each $x\in X^{\mathrm{an}}$, $|\ndot|_{\phi}(x)$ is a norm on $L\otimes_{\mathcal O_X}\kappa(x)$, which defines a morphism of sheaves (of sets) from $L$ to $j_*(\mathcal C_{X^{\mathrm{an}}}^0)$, with $\mathcal C_{X^{\mathrm{an}}}^0$ being the sheaf of continuous real functions on $X^{\mathrm{an}}$. If $L$ is an invertible $\mathcal O_X$-module equipped with a continuous metric $\phi$, for any non-zero rational section $s$ of $L$, the function $-\log|s|_\phi$, which is well defined on a Zariski open subset of $X^{\mathrm{an}}$, determines a Green function of the Cartier divisor associated with $s$. Conversely, given a Cartier divisor $D$ on $X$ equipped with a Green function of $C^0$-type $g$, the section $-D$ of $\mathscr M_X^{\times}/\mathcal O_X^{\times}$ defines an invertible sub-$\mathcal O_X$-module of $\mathscr M_X$, denoted by $\mathcal O_X(D)$, where $\mathscr M_X$ is the sheaf of rational functions on $X$. The element $-D\in\Gamma(X,\mathscr M_X^{\times}/\mathcal O_X^{\times})$ also  determines a rational section of $\mathcal O_X(D)$ denoted by $s_D$. If $f$ is a non-zero rational function of $X$ which defines the divisor $D$ on a non-empty Zariski open subset $U$, then the element $f^{-1}s_D$ is a rational section of $\mathcal O_X(D)$ which determines a regular section $s_U$ of $\mathcal O_X(D)$ on $U$ trivialising the invertible sheaf on $U$. By definition $g+\log|f|$ extends to a continuous function on $U^{\mathrm{an}}$. For any $x\in U^{\mathrm{an}}$, we let $|\ndot|_{g}(x)$ be the norm on $L\otimes_{\mathcal O_X}\kappa(x)$ such that \[|s_U(x)|_g(x)=\exp(-(g+\log|f|)(x)).\]
It does not depend on the choice of $(U,f)$. Moreover, the family of norms $(|\ndot|_g(x))_{x\in X^{\mathrm{an}}}$ defines a continuous metric on $\mathcal O_X(D)$, denoted by $\phi_g$.
\end{Remark}

\begin{Proposition}
\label{prop:exist:Green}
For any $\RR$-Cartier divisor $D$ on $X$, there is a Green function of $D$.
\end{Proposition}

\begin{proof}
First we assume that $D$ is an ample Cartier divisor.
Let $m$ be a positive integer such that $mD$ is very ample.
Let $s_0, \ldots, s_N$ be a basis of $H^0(\OO_X(mD))$ and
$\phi : X \to \PP^N = \Proj(K[T_0, \ldots, T_N])$ be the morphism given by 
\[
x \mapsto (s_0(x) : \ldots : s_N(x)).
\]
We set $z_i = T_i/T_0$ for $i=1, \ldots, N$ and $g_0 = \log \max \{ 1, |z_1|, \ldots, |z_n| \}$.
Then it is easy to see that $g_0$ is a Green function of $H_0 := \{ T_0 = 0 \}$.
Thus $\phi^*(g_0)$ is a Green function of $\phi^*(H_0)$.
We choose $\theta \in \Rat(X)^{\times}$ such that $mD = \phi^*(H_0) + (\theta)$.
Then $(1/m)(\phi^*(g_0) - \log |\theta|)$ is a Green function of $D$.

Next we assume that $D$ is a Cartier divisor.
Then there are ample Cartier divisors $A$ and $B$ with $D = A - B$.
Let $g_A$ and $g_B$ be Green functions of $A$ and $B$, respectively.
Then $g_A - g_B$ is a Green function of $D$.

In general, there are Cartier divisors $D_1, \ldots, D_r$ and
$a_1, \ldots, a_r \in \RR$ with $D = a_1 D_1 + \cdots + a_r D_r$.
Let $g_{D_i}$ be a Green function of $D_i$.
Then $a_1 g_{D_1} + \cdots + a_r g_{D_r}$ is a Green function of $D$.
\end{proof}

\begin{Proposition}\label{Pro:e-gextension}
Let $D$ be an effective  $\mathbb R$-Cartier divisor on $X$ (see Conventions and terminology~\ref{CT:R:Cartier:div}) 
and $g$ be a Green function of $D$. Then the element $\mathrm{e}^{-g}$ of $\widehat{C}^0(X)$ extends to a non-negative continuous function on $X^{\mathrm{an}}$. 
In particular, there is a constant $C$ such that $g \geqslant C$ on $X^{\mathrm{an}}$.
\end{Proposition}
\begin{proof}
Let $f$ be a local equation of $D$ on a Zariski open subset $U$ of $X$. Note that
the element $g+\log|f|$ of $\widehat{C}^0(X^{\mathrm{an}})$ extends to a continuous function on $U^{\mathrm{an}}$. Hence $\mathrm{e}^{-g}=|f|\cdot \mathrm{e}^{-(g+\log|f|)}$ extends to a continuous function on $U^{\mathrm{an}}$, which is non-negative. By gluing continuous functions we obtain that $\mathrm{e}^{-g}$ extends to a continuous function on $X^{\mathrm{an}}$.
For the last assertion, note that $X^{\mathrm{an}}$ is compact, so that there is a constant $C$ such that
$\mathrm{e}^{-g} \leqslant\mathrm{e}^{-C}$ on $X^{\mathrm{an}}$, as required.
\end{proof}

\begin{Remark}\label{Rem:effective}
Let $D$ be an effective $\mathbb R$-Cartier divisor and $g$ be a Green function of $D$. The above proposition shows that the element $\mathrm{e}^{-g}$ extends to a continuous function on $X^{\mathrm{an}}$. By abuse of notation, we use the expression $g$ to denote the map $-\log(\mathrm{e}^{-g}):X^{\mathrm{an}}\rightarrow\mathbb R\cup\{+\infty\}$, where we consider $\mathrm{e}^{-g}$ as a continuous function from $X^{\mathrm{an}}$ to $[0,+\infty[$. 
\end{Remark}

\begin{DRemark} 
\label{DRem:R:Cartier:vs:Cartier}
Let $\KK$ be either $\QQ$ or $\RR$.
In the case where $X$ is normal, for a Cartier divisor $D$ on $X$,
the effectivity of $D$ as a Cartier divisor is equivalent to the effectivity of
$D$ as a $\KK $-Cartier divisor by algebraic Hartogs' property
\footnote{If $\phi$ is a rational function on
a normal algebraic variety $V$ over a field and
$\phi$ is regular at every codimension $1$ point of $V$, then $\phi$ is regular on $V$.}.
However, 
if $X$ is not normal, then 
a Cartier divisor which is effective as a $\KK $-Cartier divisor
is not necessarily effective as a Cartier divisor. For example,
we set $X := \Proj(K[T_0, T_1, T_2]/(T_0T_2^2 - T_1^3)$, $U_i := \{ T_i \not= 0 \}\cap X$ 
($i=0,1,2$) and $x := T_1/T_0, y := T_2/T_0$ on $U_0$. 
Then $U_0 = X \setminus \{ (0:0:1) \}$ ans $U_2 = X \setminus \{ (1:0:0) \}$, so that $X = U_0 \cup U_2$.
Note that $y/x \in \OO_{X, \zeta}^{\times}$ for all $\zeta \in U_0 \cap U_2$.
Let $D$ be a Cartier divisor on $X$ given by
\[
D = \begin{cases}
(y/x) & \text{on $U_0$}, \\
(1) & \text{on $U_2$}.
\end{cases}
\]
As $y/x$ is not regular at $(1:0:0)$, $D$ is not effective as a Cartier divisor.
On the other hand, since
\[
2D = \begin{cases}
(x) & \text{on $U_0$}, \\
(1) & \text{on $U_2$},
\end{cases}
\]
$D$ is effective as a $\KK $-Cartier divisor.
As a consequence, $1 \not\in H^0(X, \OO_X(D))$ and $1 \in H^0(X, D)$, that is,
$H^0(X, \OO_X(D)) \to H^0(X, D)$ is not surjective.

\medskip
From now on, we assume that $\mathrm{char}(K) = 2$. We set $U'_0 := U_0 \setminus \{ (1 : 1 : 1) \}$.
Note that $X = U'_0 \cup U_2$ and
$1 + y/x \in \OO^{\times}_{X, \zeta}$ for all $\zeta \in U'_0 \cap U_2$, so that
we set
\[
D' := \begin{cases}
(1 + y/x) & \text{on $U'_0$}, \\
(1) & \text{on $U_2$}.
\end{cases}
\]
Since $y/x$ is not regular at $(1:0:0)$, we have $D' \not= 0$.
Moreover, as $(1 + y/x)^2 = 1 + x$, we have
\[
2D' = \begin{cases}
(1+x) & \text{on $U'_0$}, \\
(1) & \text{on $U_2$},
\end{cases}
\]
and hence $2D' = 0$ because $1 + x \in \OO_{X, \zeta}^{\times}$ for all $\zeta \in U'_0$.
Therefore, the natural homomorphism $\Div(X) \to \Div_{\KK }(X)$ is not injective.
Furthermore $\Supp_{\KK }(D') = \emptyset$, but $\Supp_{\ZZ}(D') = \{ (1:0:0) \}$.
\end{DRemark}

\subsection{Plurisubharmonic Green functions}
Let $K$ be a field equipped with a \emph{non-archimedean} complete absolute value $|\ndot|$ and $X$ be an integral projective scheme over $\Spec K$.
For each $\xi \in X^{\mathrm{an}}$, the residue field of the associated scheme point of $\xi$
is denoted by $\kappa(\xi)$.
Let $\hat{\kappa}(\xi)$ be the completion of $\kappa(\xi)$ with respect to
the absolute value $|\ndot|_\xi$ (see \S\ref{Reminder Berkovich}).
Let $L$ be an invertible sheaf on $X$.
Let $\overline{V} = (V, \|\ndot\|)$ be a finite-dimensional vector space equipped with an ultrametric norm
$\|\ndot\|$.
We assume that there is a surjective homomorphism 
$\pi : V \otimes_K \OO_X \to L$.
For each $\xi \in X^{\mathrm{an}}$,
let $\|\ndot\|_{\hat{\kappa}(\xi)}$ be the norm of $V \otimes_K \hat{\kappa}(\xi)$
obtained by the scaler extension of $\|\ndot\|$, which is by definition the operator norm on $V\otimes_K\widehat{\kappa}(\xi)\cong\Hom_K(V^\vee,\widehat{\kappa}(\xi))$ (cf. \cite[\S1.3.4]{Extension}).
The quotient norm of $L(\xi) := L \otimes_{\OO_X} \hat{\kappa}(\xi)$
induced by $\|\ndot\|_{\hat{\kappa}(\xi)}$ and the surjective
homomorphism $V \otimes_K \hat{\kappa}(\xi) \to L(\xi)$ 
is denoted by $|\ndot|^{\mathrm{quot}}_{\overline{V}}(\xi)$.
Note that $\{ |\ndot|^{\mathrm{quot}}_{\overline{V}}(\xi) \}_{\xi \in X^{\mathrm{an}}}$
yields a continuous metric on $L$ (cf. \cite[Corollary~3.4]{Extension}).

\begin{Definition}
We assume that $L$ is semiample.
A continuous metric $h = \{ |\ndot|_h(x) \}_{x\in X^{\mathrm{an}}}$ on $L$ is said to be {\em semipositive} if there are
a sequence $\{ e_n \}_{n\in\mathbb N}$ of positive integers and a sequence $\{ \overline{V}_n \}_{n\in\mathbb N}$
of normed finite-dimensional vector spaces over $K$ such that
there is a surjective homomorphism $V_n \otimes_K \OO_X \to L^{\otimes e_n}$ for each $n$ and
the sequence
\[
\left\{ \frac{1}{e_n} \log \frac{|\ndot|^{\mathrm{quot}}_{\overline{V}_n}(\xi)}{|\ndot|_{h^{e_n}}(\xi)}
\right\}_{n\in\mathbb N}
\]
converges to $0$ uniformly on $X^{\mathrm{an}}$.
In other words,
if we choose a non-zero rational section $s$ of $L$, then
the sequence 
\[
\left\{ \frac{1}{e_n} \log |s^{e_n} |^{\mathrm{quot}}_{\overline{V}_n}(\xi) \right\}_{n \in\mathbb N}
\]
of $\zeros(s)$-Green functions
converges to $\log |s|_h(\xi)$ uniformly (cf. Definition~\ref{Def:Green functions}).
\end{Definition}

We recall a characterisation of semipositive metrics as follows.

\begin{Proposition}[{\cite[Corollary~3.11]{Extension}}]
\label{prop:characterization:semipositive}
Let $L$ be a semiample invertible sheaf on $X$ and $h$ be a continuous metric on $L$.
Then $h$ is semipositive if and only if, for any $\epsilon > 0$,
there is a positive integer $n$ such that, for any $\xi \in X^{\mathrm{an}}$,
we can find $s \in H^0(X, L^n)_{\hat{\kappa}(\xi)} \setminus \{ 0 \}$
with $\| s \|_{h^n, \hat{\kappa}(\xi)} \leqslant e^{n\epsilon} |s|_{h^n}(\xi)$.
\end{Proposition}

\begin{Proposition}
\label{prop:basic:prop:semipositive}
Let $L$ and $L'$
be semiample invertible sheaves on $X$ and $h$ and $h'$ be continuous metrics on $L$ and $L'$, respectively.
\begin{enumerate}
\renewcommand{\labelenumi}{(\arabic{enumi})}
\item
If $h$ and $h'$ are semipositive, then
the metric $h \otimes h'$ on $L\otimes L'$ is also semipositive.

\item
Let $f : Y \to X$ be a morphism of projective integral schemes over $\Spec K$.
If $h=\{|\ndot|_h(x)\}_{x\in X^{\mathrm{an}}}$ is semipositive, then
$f^*(L)$ is semiample and
$(f^{\mathrm{an}})^*(h)=\{|\ndot|_h(f^{\mathrm{an}}(y))_{\hat{\kappa}(y)}\}_{y\in Y^{\mathrm{an}}}$ is a semipositive metric on $f^*(L)$, where $|\ndot|_h(f^{\mathrm{an}}(y))_{\hat{\kappa}(y)}$ denotes the norm on $f^*(L)\otimes_{\mathcal O_Y}\hat{\kappa}(y)\cong L(x)\otimes_{\hat{\kappa}(x)}\hat{\kappa}(y)$ induced by $|\ndot|_h(f^{\mathrm{an}}(y))$ by extension of scalars.

\item
Let $\{ h_n \}_{n=1}^{\infty}$ be a sequence of semipositive metrics of $L^{\mathrm{an}}$.
If
\[
\left\{ \log \frac{|\ndot|_{h_n}}{|\ndot|_h} \right\}_{n=1}^{\infty}
\]
converges to $0$ uniformly, then $h$ is semipositive.

\item
The following are equivalent:
\begin{enumerate}
\renewcommand{\labelenumii}{(\arabic{enumi}.\arabic{enumii})}
\item
$h$ is semipositive.

\item
$h^n$ is semipositive for all $n \geqslant 1$.

\item
$h^n$ is semipositive for some $n \geqslant 1$.
\end{enumerate}
\end{enumerate}
\end{Proposition}

\begin{proof}
(1) As $h$ and $h'$ are semipositive,
by Proposition~\ref{prop:characterization:semipositive},
for any $\epsilon > 0$,
there are positive integers $n$ and $n'$ such that,
for all $\xi \in X^{\mathrm{an}}$,
we can find
\[
s \in H^0(X, L^n)_{\hat{\kappa}(\xi)} \setminus \{ 0 \}
\quad\text{and}\quad 
s' \in H^0(X, L^{n'})_{\hat{\kappa}(\xi)} \setminus \{ 0 \}
\]
with 
\[
\| s \|_{h^n, \hat{\kappa}(\xi)} \leqslant e^{n\epsilon} |s|_{h^n}(\xi)
\quad\text{and}\quad
\| s' \|_{h^{n'}, \hat{\kappa}(\xi)} \leqslant e^{n'\epsilon} |s'|_{h^{n'}}(\xi).
\]
Then $s^{n'} {s'}^{n} \in H^0(X, (L\otimes L')^{nn'})_{\hat{\kappa}(\xi)} \setminus \{ 0 \}$
and
\begin{align*}
\| s^{n'} {s'}^{n} \|_{(h \otimes h)^{nn'}, \hat{\kappa}(\xi)} & \leqslant
\left(\| s \|_{h^{n}, \hat{\kappa}(\xi)}\right)^{n'} \left(\| s' \|_{h^{n'}, \hat{\kappa}(\xi)}\right)^{n} \\
& \leqslant e^{nn'\epsilon} (|s|_{h^{n}}(\xi))^{n'}(|s'|_{h^{n'}}(\xi))^n \\
& = e^{nn'\epsilon} |s^{n'} {s'}^n|_{(h\otimes h')^{nn'}}(\xi).
\end{align*}
Therefore, by Proposition~\ref{prop:characterization:semipositive} again,
$h \otimes h'$ is semipositive.

\medskip
(2) The semiampleness of $f^*(L)$ is obvious. By Proposition~\ref{prop:characterization:semipositive},
for any $\epsilon > 0$, there is a positive integer $n$ such that,
for any $\zeta \in Y^{\mathrm{an}}$, we can find $s \in H^0(X, L^{n})_{\hat{\kappa}(f^{\mathrm{an}}(\zeta))} \setminus \{ 0 \}$
with $\| s \|_{h^n, \hat{\kappa}(f^{\mathrm{an}}(\zeta))} \leqslant e^{n\epsilon} |s|_{h^n}(f^{\mathrm{an}}(\zeta))$.
Then, as $s$ is not zero at the scheme point of $f^{\mathrm{an}}(\zeta)$,
\[
f^*_{\hat{\kappa}(f^{\mathrm{an}}(\zeta))}(s) \in H^0(Y, f^*(L))_{\hat{\kappa}(f^{\mathrm{an}}(\zeta))} \setminus \{ 0 \},
\]
where $f^*_{\hat{\kappa}(f^{\mathrm{an}}(\zeta))} : H^0(X, L)_{\hat{\kappa}(f^{\mathrm{an}}(\zeta))}
\to H^0(Y, f^*(L))_{\hat{\kappa}(f^{\mathrm{an}}(\zeta))}$
is the natural homomorphism.
Thus, if we set 
\[
s' = f^*_{\hat{\kappa}(f^{\mathrm{an}}(\zeta))}(s) \otimes_{\hat{\kappa}(f^{\mathrm{an}}(\zeta))} 1_{\hat{\kappa}(\zeta)} \in 
H^0(Y, f^*(L))_{\hat{\kappa}(\zeta)} \setminus \{ 0 \},
\]
Then
\begin{align*}
\| s' \|_{(f^{\mathrm{an}})^*(h^n), \hat{\kappa}(\zeta)} & = \|
f^*_{\hat{\kappa}(f^{\mathrm{an}}(\zeta))}(s)  
\|_{(f^{\mathrm{an}})^*(h^n), \hat{\kappa}(f^{\mathrm{an}}(\zeta))} \\
& \leqslant \| s \|_{h^n, \hat{\kappa}(f^{\mathrm{an}}(\zeta))} \leqslant e^{n\epsilon} |s|_{h^n}(f^{\mathrm{an}}(\zeta)) \\
& = e^{n\epsilon} |s'|_{(f^{\mathrm{an}})^*(h^n)}(\zeta),
\end{align*}
so that the assertion follows from Proposition~\ref{prop:characterization:semipositive}.

\medskip
(3) For $\epsilon > 0$, there is a positive integer $n_0$ such that
\[
e^{-\epsilon} \leqslant \frac{|\ndot|_{h_{n_0}}}{|\ndot|_h} \leqslant e^{\epsilon}\quad\text{on $X^{\mathrm{an}}$}.
\]
Moreover, as $h_{n_0}$ is semipositive, 
there is a positive integer $n_1$ such that, for any $\xi \in X^{\mathrm{an}}$,
we can find $s \in H^0(X, L^{n_1})_{\hat{\kappa}(\xi)} \setminus \{ 0 \}$ with
$\| s \|_{h_{n_0}^{n_1}, \hat{\kappa}(\xi)} \leqslant e^{n_1 \epsilon} |s|_{h_{n_0}^{n_1}}(\xi)$, and hence
\begin{align*}
\| s\|_{h^{n_1}, \hat{\kappa}(\xi)} \leqslant e^{n_1\epsilon}  \| s \|_{h_{n_0}^{n_1}, \hat{\kappa}(\xi)} \leqslant 
e^{2 n_1\epsilon} |s|_{h_{n_0}^{n_1}}(\xi) \leqslant e^{3 n_1\epsilon} |s|_{h^{n_1}}(\xi),
\end{align*}
so that $h$ is semipositive by Proposition~\ref{prop:characterization:semipositive}.

\medskip
(4) ``(4.1) $\Longrightarrow$ (4.2)'' is a consequence of (1).
``(4.2) $\Longrightarrow$ (4.3)'' is obvious.
We can easily check ``(4.3) $\Longrightarrow$ (4.1)'' by using Proposition~\ref{prop:characterization:semipositive}.
\end{proof}

\begin{Definition}
Let $D$ be a semiample $\QQ$-Cartier divisor on $X$ and $g$ be a $D$-Green function of $C^0$-type.
We say that $g$ is \emph{of plurisubharmonic type} (or \emph{plurisubharmonic}) if there is a positive integer $n$ such that
$nD$ is a Cartier divisor and $|\ndot|_{ng}$ is a semipositive metric of $\OO_X(nD)$.
Note that, by (4) in Proposition~\ref{prop:basic:prop:semipositive},
the last condition does not depend on the choice of $n$.
Moreover, if $nD$ is a Cartier divisor for some positive integer $n$, then $|\ndot|_{ng}$ is semipositive.
\end{Definition}

\begin{Proposition}[$\QQ$-version]
\label{prop:basic:prop:pluri:Q}
Let $D$ and $D'$ be semiample $\QQ$-Cartier divisors on $X$, and let $g$ and $g'$ be
plurisubharmonic Green functions of $D$ and $D'$, respectively.
Then we have the following:
\begin{enumerate}
\renewcommand{\labelenumi}{(\arabic{enumi})}
\item For $\phi \in \Rat(X)^{\times}_{\QQ}$, $-\log |\phi|$ is a $(\phi)$-Green function
of plurisubharmonic type.

\item
For all $a, a' \in \QQ_{>0}$, $ag + a'g'$ is also of plurisubharmonic type.

\item
Let $f : Y \to X$ be a morphism of projective integral schemes over $K$ such that
$f(Y) \not\subseteq \Supp(D)$.
Then $f^*(D)$ is semiample and
$(f^{\mathrm{an}})^*(g)$ is an $f^*(D)$-Green function of plurisubharmonic type.

\item
Let $\{ g_n \}_{n\in\mathbb N}$ be a sequence of $D$-Green functions of plurisubharmonic type.
If $\{ g_n \}_{n\in\mathbb N}$ converges a $D$-Green function $g$ uniformly
(cf. Definition \ref{Def:Green functions}), then
$g$ is also of plurisubharmonic type.
\end{enumerate}
\end{Proposition}

\begin{proof}
(1) Clearly we may assume that $\phi \in \Rat(X)^{\times}$.
We set $D = (\phi)$ and $g = -\log |\phi|$.
As $\OO_X(D) = \OO_X \phi^{-1}$, the homomorphism $\OO_X \to \OO_X(D)$
given by $1 \mapsto \phi^{-1}$ yields an isomorphism.
If we give the trivial metric $|\ndot|$ to $\OO_X$,
then the above isomorphism gives rise to
an isometry $(\OO_X, |\ndot|) \simeq (\OO_X(D), |\ndot|_g)$.
Thus the assertion follows.

(2), (3) and (4) follows from (1), (2), (3) in Proposition~\ref{prop:basic:prop:semipositive}.
\end{proof}

\begin{Definition}
Let $D$ be an $\RR$-Cartier divisor on $X$.
We assume that $D$ is semiample, that is,
there are semiample Cartier divisors $A_1, \ldots, A_r$ and
$a_1, \ldots, a_r \in \RR_{>0}$ such that $D= a_1A_1 + \cdots + a_r A_r$
(cf. Conventions and terminology~\ref{CT:semiample}).
We say a $D$-Green function $g$ of $C^0$-type is said to be
{\em of plurisubharmonic type} (or plurisubharmonic) if there is a sequence $\{ g_n \}_{n\in\mathbb N}$ of
$D$-Green functions of $C^0$-type with the following conditions:
\begin{enumerate}
\renewcommand{\labelenumi}{(\arabic{enumi})}
\item
Let $\theta_n$ be the continuous extension of $g - g_n$ on $X^{\mathrm{an}}$.
Then $\lim_{n\to\infty} \| \theta_n \|_{\sup} = 0$.

\item
For each $n$, there are semiample $\QQ$-Cartier divisors $A_{n1}, \ldots, A_{nr_n}$ on $X$,
plurisubharmonic Green function $g_{n1}, \ldots, g_{nr_n}$ of $A_{n1}, \ldots, A_{nr_n}$, respectively and
positive real numbers $a_{n1}, \ldots, a_{nr_n}$ such that $D = a_{n1}A_{n1} + \cdots + a_{nr_n} A_{nr_n}$ and
$g_n = a_{n1}g_{n1} + \cdots + a_{nr_n} g_{nr_n}$.
\end{enumerate}
We refer the readers to \cite[\S3]{Gubler_Kunnemann} for more details about plurisubhamonic functions and semi-positive metrics.
\end{Definition}

The $\RR$-version of Proposition~\ref{prop:basic:prop:pluri:Q} can be checked by
using the $\QQ$-version.

\begin{Proposition}[$\RR$-version]
\label{prop:basic:prop:pluri:R}
Let $D$ and $D'$ be semiample $\RR$-Cartier divisors on $X$, and let $g$ and $g'$ be
plurisubharmonic Green functions of $D$ and $D'$, respectively.
Then we have the following:
\begin{enumerate}
\renewcommand{\labelenumi}{(\arabic{enumi})}
\item For $\phi \in \Rat(X)^{\times}_{\RR}$, $-\log |\phi|$ is a $(\phi)$-Green function
of plurisubharmonic type.

\item
For all $a, a' \in \RR_{>0}$, $ag + a'g'$ is also of plurisubharmonic type.

\item
Let $f : Y \to X$ be a surjective morphism of projective integral schemes over $K$
such that $f(Y) \not\subseteq \Supp(D)$.
Then $f^*(D)$ is semiample and
$(f^{\mathrm{an}})^*(g)$ is an $f^*(D)$-Green function of plurisubharmonic type.

\item
Let $\{ g_n \}$ be a sequence of $D$-Green functions of plurisubharmonic type.
If $\{ g_n \}$ converges a $D$-Green function $g$ uniformly
(cf. Definition \ref{Def:Green functions}), then
$g$ is also of plurisubharmonic type.
\end{enumerate}
\end{Proposition}

Finally let us see the following proposition:

\begin{Proposition}
\label{prop:Q:PSH:equiv:R:PSH}
Let $D$ be a semiample $\QQ$-Cartier divisor on $X$ and $g$ be a $D$-Green function of $C^0$-type.
The Green function $g$ is of plurisubharmonic type as a $\QQ$-Cartier divisor if and only if
$g$ is of plurisubharmonic type as an $\RR$-Cartier divisor.
\end{Proposition}

\begin{proof}
It is sufficient to show that if $g$ is plurisubharmonic as an $\RR$-Cartier divisor, then
$g$ is plurisubharmonic as a $\QQ$-Cartier divisor.
By (4) in Proposition~\ref{prop:basic:prop:pluri:Q},
we may assume that $g$ is obtained by the following way: 
there are semiample $\QQ$-Cartier divisors
$A_1, \ldots, A_r$ on $X$ and plurisubharmonic Green functions
$h_1, \ldots, h_r$ of $A_1, \ldots, A_r$, respectively
such that 
\[
D = a_1 A_1 + \cdots + a_r A_r
\quad\text{and}\quad 
g = a_1 h_1 + \cdots + a_r h_r
\]
for some $a_1, \ldots, a_r \in \RR_{>0}$. 
Here we claim the following:

\begin{Claim}
\label{claim:prop:Q:PSH:equiv:R:PSH:01}
Let $V$ be a vector space over $\QQ$. Then we have the following:
\begin{enumerate}
\renewcommand{\labelenumi}{(\roman{enumi})}
\item
$W_{\RR} \cap V = W$ for any a vector subspace $W$ of $V$.

\item
Let $x, x_1, \ldots, x_r \in V$ such that
$x = a_1 x_1 + \cdots + a_r x_r$
for some $a_1, \ldots, a_r \in \RR$. Then, for any $\epsilon > 0$,
there are $a'_1, \ldots, a'_r \in \QQ$ such that
$x = a'_1 x_1 + \cdots + a'_r x_r$ and $| a'_i - a_i | \leqslant \epsilon$ for all $i$.
\end{enumerate}
\end{Claim}

\begin{proof}
(i) is obvious because $V/W \to (V/W)_{\RR}$ is injective and $(V/W)_{\RR} = V_{\RR}/W_{\RR}$.

\smallskip
(ii)
We set $W := \QQ x_1 + \cdots + \QQ x_r$. Then, by (i),
there are $b_1, \ldots, b_r \in \QQ$ such that
$x = b_1 x_1 + \cdots + b_r x_r$.
Let us consider a homomorphism $\psi : \QQ^r \to V$ given by
$\psi(t_1, \ldots, t_r) = t_1 x_1 + \cdots + t_r x_r$.
We denote the scalar extension $\RR^r \to V_{\RR}$ by $\psi_{\RR}$, that is,
$\psi_{\RR}(\alpha_1, \ldots, \alpha_r) = \alpha_1 x_1 + \cdots + \alpha_r x_r$.
We set
\[
\delta := (a_1, \ldots, a_r) - (b_1, \ldots, b_r) \in \Ker(\psi_{\RR}).
\]
As $\Ker(\psi_{\RR}) = \Ker(\psi)_{\RR}$, $\Ker(\psi)$ is dense in $\Ker(\psi_{\RR})$, so that
there is $\delta' \in \Ker(\psi)$ 
such that $| \delta - \delta'| \leqslant \epsilon$,
where for $y = (y_1, \ldots, y_r) \in \RR^r$,
$|y| := \max \{ |y_1|, \ldots, |y_r| \}$.
Therefore, if we set 
\[
(a'_{1}, \ldots, a'_{r}) = (b_1, \ldots, b_r) + \delta',
\]
then $x = a'_{1} x_1 + \cdots + a'_{r} x_r$ ($a'_{1}, \ldots, a'_{r} \in \QQ$) and
$|a_{i} - a'_{i}| \leqslant \epsilon$ for all $i=1, \ldots, r$.
\end{proof}

By applying the above claim to the case where $V = \Div_{\QQ}(X)$,
$x = D$ and $x_i = A_i$ ($i=1, \ldots, r$),
there are sequences $\{ a_{n1} \}_{n=1}^{\infty}, \ldots, \{ a_{nr} \}_{n=1}^{\infty}$
of positive rational numbers such that 
\[
a_i = \lim_{n\to\infty} a_{ni}\quad (i=1, \ldots, r)
\quad\text{and}\quad
D = a_{n1}A_1 + \cdots + a_{nr} A_r.
\]
We set $g_n := a_{n1}h_1 + \cdots + a_{nr} h_r$. 
Then $g_n$ is a $D$-Green function of plurisubharmonic type by (2) in Proposition~\ref{prop:basic:prop:pluri:Q}.
Let $\theta_n$ be a continuous function on $X^{\mathrm{an}}$ with
$g - g_n = \theta_n$. It is sufficient to see that $\lim_{n\to\infty}\| \theta_n \|_{\sup} = 0$
by virtue of (4) in Proposition~\ref{prop:basic:prop:pluri:Q}.
If we set $b_{ni} := a_i - a_{ni}$, then $\lim_{n\to\infty} b_{ni} = 0$, $b_{n1} A_1 + \cdots + b_{nr} A_r = 0$ and
the continuous extension of 
$b_{n1} h_1 + \cdots + b_{nr} h_r$ is $\theta_n$.
Let $E_1, \ldots, E_s$ be a basis of the vector subspace $\QQ A_1 + \cdots + \QQ A_r$ of
$\Div_{\QQ}(X)$. We choose $\alpha_{i1}, \ldots, \alpha_{is} \in \QQ$ with
$A_i = \sum_{j=1}^s \alpha_{ij} E_j$.
Then, as
\[
0 = \sum_{i} b_{ni} A_i = \sum_{j=1}^s \left(\sum_{i=1}^r b_{ni}\alpha_{ij}\right) E_j,
\]
we have $\sum_{i=1}^r b_{ni}\alpha_{ij} = 0$ for all $n \geqslant 1$ and $j=1, \ldots, s$.
For each $j$, let $e_j$ be an $E_j$-Green function of $C^0$-type.
Then, for each $i$, there is a continuous function $\nu_i$ on $X^{\mathrm{an}}$
such that $h_i - \sum_{j=1}^s\alpha_{ij} e_j = \nu_i$. Note that
\begin{align*}
\theta_n & = \sum_{i=1}^r b_{ni} h_i = \sum_{i=1}^r b_{ni} \left(  \nu_i + \sum_{j=1}^s \alpha_{ij} e_j \right) \\
& = \sum_{i=1}^r b_{ni}\nu_i + \sum_{j=1}^s \left(\sum_{i=1}^r b_{ni} \alpha_{ij}\right)e_j 
= \sum_{i=1}^r b_{ni}\nu_i.
\end{align*}
Thus $\| \theta_n \|_{\sup} \leqslant \sum_{i=1}^r |b_{ni}| \| \nu_i \|_{\sup}$,
and hence the assertion follows.
\end{proof}

\subsection{Canonical Green functions with respect to endomorphisms}
\label{subsec:canonical:green:function}
Given a polarised dynamic system on a projective variety over $\Spec K$ one can attach to the polarisation divisor a canonical Green function, which is closely related to the canonical local height function. We refer the readers to \cite{Neron65} for the original work of N\'eron in the Abelian variety case, and to \cite{Call_Silverman93,Zhang95} for general dynamic systems in the setting of canonical local height and canonical metric respectively. See \cite{Kawaguchi_Silverman09} for the non-archimedean case. In the following, we recall the construction of the canonical Green functions of $\mathbb R$-Cartier divisors.

Let $f : X \to X$ be a surjective endomorphism of $X$ over $K$.
Let $D$ be an $\RR$-Cartier divisor on $X$.
We assume that there are a real number $d$ and $\varphi \in \Rat(X)^{\times}_{\RR}$ such that
$d > 1$ and $f^*(D) = dD + (\varphi)$.
We fix a Green function $g_0$ of $D$.
There exists a unique continuous function $\lambda$ on $X^{\mathrm{an}}$ such that
\[
(f^{\mathrm{an}})^*(g_0) = d g_0 - \log |\varphi|  + \lambda,
\]
where for any element $g\in\widehat{C}^0(X^{\mathrm{an}})$ represented by a continuous function $h:U^{\mathrm{an}}\rightarrow\mathbb R$, with $U$ being a non-empty Zariski open subset of $X$, the expression $(f^{\mathrm{an}})^*(g)$ denotes the element in $\widehat{C}^0(X^{\mathrm{an}})$ represented by the function $h\circ f^{\mathrm{an}}: f^{-1}(U)^{\mathrm{an}}\rightarrow\mathbb R$.
We set
\begin{equation}
\label{eqn:prop:unique:existence:canonical:Green:function:01}
h_n = \sum_{i=0}^{n-1} \frac{1}{d^{i+1}}((f^{\mathrm{an}})^i)^*(\lambda)\quad(n \geqslant 1).
\end{equation}

\begin{Lemma}
\label{lem:uniform:conv}
The sequence $\{ h_n \}_{n\geqslant 1}$ of continuous functions on $X^{\mathrm{an}}$
converges to a continuous function $h$ on $X^{\mathrm{an}}$ uniformly.
\end{Lemma}

\begin{proof}
If $n > m$, then
\begin{align*}
\| h_n - h_m \|_{\sup} & \leqslant \sum_{i=m}^{n-1} \frac{1}{d^{i+1}}
\|((f^{\mathrm{an}})^i)^*(\lambda) \|_{\sup} = \frac{\| \lambda \|_{\sup} }{d^{m+1}}
\sum_{i=0}^{n-m-1} \frac{1}{d^{i}} \\
& \leqslant \frac{\| \lambda \|_{\sup} }{d^{m+1}} \sum_{i=0}^{\infty} \frac{1}{d^{i}} =
\frac{\| \lambda \|_{\sup}}{d^m(d - 1)}.
\end{align*}
Thus the lemma follows.
\end{proof}

\begin{Proposition}
\label{prop:unique:existence:canonical:Green:function}
There is a unique Green function of $D$ with $(f^{\mathrm{an}})^*(g) = d g - \log |\varphi|$ on $X^{\mathrm{an}}$.
\end{Proposition}

\begin{proof}
Let us begin with the uniqueness of $g$.
Let $g'$ be another Green function of $D$ with $(f^{\mathrm{an}})^*(g') = d g' - \log |\varphi|$ on $X^{\mathrm{an}}$.
Then $(f^{\mathrm{an}})^*(g'-g) = d(g' - g)$ on $X^{\mathrm{an}}$. Note that there is a continuous function $\theta$ on
$X^{\mathrm{an}}$ with $\theta = g' - g$, so that $(f^{\mathrm{an}})^*(\theta) = d(\theta)$. Here we consider the sup norm
$\|\ndot\|_{\sup}$ of continuous functions. Then
\[
\| \theta \|_{\sup} = \| (f^{\mathrm{an}})^*(\theta) \|_{\sup} = \| d\theta \|_{\sup} = d \| \theta \|_{\sup},
\]
and hence $\| \theta \|_{\sup} = 0$. Therefore, $\theta = 0$.

\medskip
Since
\[
(f^{\mathrm{an}})^*(h_n) = \sum_{i=0}^{n-1} \frac{1}{d^{i+1}}((f^{\mathrm{an}})^{i+1})^*(\lambda) = d h_{n+1} - \lambda,
\]
we have $(f^{\mathrm{an}})^*(h) = d h - \lambda$, so that if we $g = g_0 + h$, then
\[
(f^{\mathrm{an}})^*(g) =  (d g_0 - \log |\varphi|  + \lambda) + (d h - \lambda) = d g - \log |\varphi|,
\]
as required.
\end{proof}

A Green function $g$ of $D$ is called the \emph{canonical Green function} of $D$ with respect to $f$ if
$(f^{\mathrm{an}})^*(g) = d g - \log |\varphi|$ on $X^{\mathrm{an}}$.

\begin{Lemma}
\label{lem:canonical:Green:principal}
For $\theta \in \Rat(X)^{\times}_{\RR}$, we have the following:
\begin{enumerate}
\item
$f^*(D + (\theta)) = d(D + (\theta)) + \left(f^*(\theta)\theta^{-d}\varphi\right)$.

\item
The canonical Green function of $D + (\theta)$ is given by
$g - \log | \theta |$.
\end{enumerate}
\end{Lemma}

\begin{proof}
(1) is obvious. Since
\begin{align*}
(f^{\mathrm{an}})^*(g - \log | \theta |) & = d g - \log | \varphi | - \log |f^*(\theta))| \\
& = d(g  - \log |\theta|) - \log |f^*(\theta)\theta^{-d}\varphi|,
\end{align*}
the assertion (2) follows.
\end{proof}

We set
\begin{equation}
\begin{cases}
g_n := g_0 + h_n &  (n \geqslant 1)\\
{\displaystyle \varphi_0 = 1, \quad \varphi_n = \prod_{i=0}^{n-1} (f^i)^*(\varphi)^{1/d^{i+1}}} &  (n \geqslant 1).
\end{cases}
\end{equation}
Let us see the following facts:

\begin{Lemma}
\label{lem:recursive:relations}
\begin{enumerate}
\renewcommand{\labelenumi}{\rom{(\arabic{enumi})}}
\item $(f^{\mathrm{an}})^*(g_{n-1}) = d g_n  - \log |\varphi|$ and
$f^*(\varphi_{n-1}) = \varphi_n^d/\varphi$ for all $n \geqslant 1$.

\item
If $D \geqslant 0$ and $g_0 \geqslant 0$, then $D + (\varphi_n) \geqslant 0$ and $g_{n}  - \log |\varphi_n| \geqslant 0$ for all $n \geqslant 0$.
\end{enumerate}
\end{Lemma}

\begin{proof}
(1) In the case $n=1$, since 
$(f^{\mathrm{an}})^*(g_0) = dg_0 - \log |\varphi| + \lambda$,
$g_1 = g_0 + (1/d)\lambda$ and $\varphi_1 = \varphi^{1/d}$,
the assertion is obvious.
For $n \geqslant 2$, 
{\allowdisplaybreaks
\begin{align*}
(f^{\mathrm{an}})^*(g_{n-1}) & = (f^{\mathrm{an}})^*\left(g_0 + \sum_{i=0}^{n-2} \frac{1}{d^{i+1}}((f^{\mathrm{an}})^i)^*(\lambda)\right) \\
& = (f^{\mathrm{an}})^*(g_0) + \sum_{i=0}^{n-2} \frac{1}{d^{i+1}}((f^{\mathrm{an}})^{i+1})^*(\lambda) \\
& = d g_0 - \log |\varphi| + \lambda + \sum_{i=0}^{n-2} \frac{1}{d^{i+1}}((f^{\mathrm{an}})^{i+1})^*(\lambda) \\
& = d\left( g_0 + \frac{1}{d} \lambda + \sum_{i=0}^{n-2} \frac{1}{d^{i+2}}((f^{\mathrm{an}})^{i+1})^*(\lambda)\right) - \log |\varphi| \\
& = d g_n  - \log |\varphi| \\
\intertext{and}
f^*(\varphi_{n-1}) & = f^*\left(\prod_{i=0}^{n-2} (f^i)^*(\varphi)^{1/d^{i+1}}\right) 
= \prod_{i=0}^{n-2} (f^{i+1})^*(\varphi)^{1/d^{i+1}} \\
& = \prod_{i=1}^{n-1} (f^{i})^*(\varphi)^{1/d^{i}} 
= \left( \prod_{i=1}^{n-1} (f^{i})^*(\varphi)^{1/d^{i+1}} \right)^{d} = (\varphi_n/\varphi^{1/d})^d
= \varphi_n^d/\varphi.
\end{align*}}
Therefore, the assertion follows inductively.

\medskip
(2) follows from (1).
\end{proof}

\begin{Proposition}
\label{prop:canonical:Dirichlet}
If $D + (s)$ is effective for some $s \in \Rat(X)^{\times}_{\RR}$, then, for any $\epsilon > 0$,
there is $\varphi_{\epsilon} \in \Rat(X)^{\times}_{\RR}$ such that
$D + (\varphi_{\epsilon}) \geqslant 0$ and $g - \log |\varphi_{\epsilon}| + \epsilon \geqslant 0$.
\end{Proposition}

\begin{proof}
First we assume that $D \geqslant 0$.
By Proposition~\ref{prop:exist:Green} and Proposition~\ref{Pro:e-gextension},
we can choose a Green function $g_0$ of $D$ with $g_0 \geqslant 0$. Then, by Lemma~\ref{lem:uniform:conv} 
and Lemma~\ref{lem:recursive:relations},
for $\epsilon > 0$, there is a positive integer $n$ such that $\| h - h_n \|_{\sup} \leqslant \epsilon$,
$D + (\varphi_n) \geqslant 0$ and $g_n - \log |\varphi_n| \geqslant 0$, and hence
\[
0 \leqslant g_n - \log |\varphi_n| = g + (h_n - h) - \log |\varphi_n| \leqslant g + \epsilon - \log |\varphi_n|,
\]
as required.

Next we assume that $D + (s)$ is effective for some $s \in \Rat(X)^{\times}_{\RR}$.
Then, by Lemma~\ref{lem:canonical:Green:principal},
the canonical Green function of $D + (s)$ is
$g - \log |s|$. Therefore, by the previous observation,
for any $\epsilon > 0$, there is $\psi_{\epsilon} \in \Rat(X)^{\times}_{\RR}$ 
such that $D + (s) + (\psi_{\epsilon}) \geqslant 0$ and
$g - \log |s| - \log |\psi_{\epsilon}| + \epsilon \geqslant 0$, and hence
the assertion of the proposition follows for $\varphi_{\epsilon} := s \psi_{\epsilon}$.
\end{proof}

\begin{Proposition}
\label{prop:semiample:plurisubharmonic:endomorphism}
If $D$ is semiample, then the canonical Green function of $D$ is of plurisubharmonic type.
\end{Proposition}

\begin{proof}
Since $D$ is semiample, we can choose a Green function $g_0$ of plurisubharmonic type as
an initial Green function.
Note that $g$ is the uniform limit of the sequence $\{ g_n \}_{n\geqslant 1}$, so that,
by (4) in Proposition~\ref{prop:basic:prop:pluri:R}, 
it is sufficient to show that each $g_n$ is of plurisubharmonic type.
We prove it by induction on $n$.
We assume that $g_{n-1}$ is of plurisubharmonic type. Then,
by Lemma~\ref{lem:recursive:relations} and (3) in Proposition~\ref{prop:basic:prop:pluri:R}, 
$d g_n  - \log |\varphi|$ is of plurisubharmonic type.
Therefore, by (1) and (2) Proposition~\ref{prop:basic:prop:pluri:R},
$g_n$ is of plurisubharmonic type.
\end{proof}

\section{A sufficient condition for the Dirichlet property \\
of arithmetic dynamic system}

Throughout this section,
let $K$ be a number field and $X$
be a normal, projective and geometrically integral scheme over $\Spec K$.
In this section, we consider a sufficient condition to guarantee the Dirichlet property.
As a consequence, if $X = \PP^n_K$ and $f$ is an endomorphism given by polynomials, then
we can see that the Dirichlet property holds for the canonical compatification of
an ample Cartier divisor with respect to the endomorphism $f$.

\subsection{Preliminaries}
In this subsection, we discuss
several facts which are used in the later subsection.

\begin{Lemma}
\label{lemma:injective:principal}
The natural homomorphism $\Rat(X)^{\times}_{\RR} \to 
\aDiv_{\RR}(X)$  
given by $\varphi \mapsto \widehat{(\varphi)}$
is injective, that is, $\Rat(X)^{\times}_{\RR}$ can be considered as a vector subspace of 
$\aDiv_{\RR}(X)$. 
\end{Lemma}

\begin{proof}
We denote the homomorphism $\Rat(X)^{\times}_{\RR} \to 
\aDiv_{\RR}(X)$ 
by $\alpha$.
Let $\varphi \in \Ker(\alpha)$. We set $\varphi = \varphi_1^{a_1} \cdots \varphi_l^{a_l}$
such that $\varphi_1, \ldots, \varphi_l \in \Rat(X)^{\times}$,
$a_1, \ldots, a_l \in \RR$ and $a_1, \ldots, a_l$ are linearly independent over $\QQ$.
As $a_1(\varphi) + \cdots + a_l (\varphi_l)$ vanishes in $\widehat{\mathrm{Div}}_{\mathbb R}(X)$, for any prime divisor $\Gamma$ on $X$,
\[
a_1 \ord_{\Gamma}(\varphi_1) + \cdots + a_l \ord_{\Gamma}(\varphi_l) = 0,
\]
which implies $\ord_{\Gamma}(\varphi_i) = 0$ for all $i$ because
$a_1, \ldots, a_l$ are linearly independent over $\QQ$.
Thus $\varphi_i \in K^{\times}$ for all $i$, by which we may assume that $X = \Spec(K)$. 
For $v \in M_K^{\mathrm{fin}}$, as before,
\[
a_1 \ord_{v}(\varphi_1) + \cdots + a_l \ord_{v}(\varphi_l) = 0,
\]
so that $\varphi_i \in O_K^{\times}$ for all $i$.
Therefore, $\Ker(\alpha) \subseteq (O_K^{\times})_{\RR}$.
Let us consider the homomorphism $L : O_K^{\times} \to \RR^{K(\CC)}$
given by $x \mapsto (-\log|x|)$, where $(-\log|x|)_{\sigma} = -\log|\sigma(x)|$.
It is well known that $\Ker(L)$ is a finite group, so that
the natural extension $L_{\RR} : (O_K^{\times})_{\RR} \to \RR^{K(\CC)}$ is injective.
Therefore we have the assertion.
\end{proof}

\begin{Lemma}
\label{lemma:limit:effective}
Let $H$ be a finite dimensional vector subspace of
$\aDiv_{\RR}(X)$ 
over $\RR$.
Let $\{ \overline{D}_n \}_{n\in\mathbb N}$ be a sequence in $H$.
Moreover, 
let $\{(0, \theta_{n}) \}_{n\in\mathbb N}$ be a sequence in 
$\aDiv_{\RR}(X)$, 
that is,
for each $v \in M_K$, $\{ \theta_{n,v} \}_{n\in\mathbb N}$ is a sequence of continuous functions on $X_v^{\mathrm{an}}$.
We assume the following:
\begin{enumerate}
\renewcommand{\labelenumi}{\rom{(\arabic{enumi})}}
\item $\{ \overline{D}_n \}_{n\in\mathbb N}$ has a limit
$\overline{D}$ in the natural topology of $H$ as a finite dimensional vector space over $\RR$.

\item
For each $v \in \Sigma$, $\{\theta_{n,v} \}_{n\in\mathbb N}$ converges to $0$ uniformly.
\end{enumerate}
If $\overline{D}_n + \left(0, \theta_n \right)\geqslant 0$ for all $n$, then $\overline{D} \geqslant 0$.
\end{Lemma}

\begin{proof}
Let $\overline{H}_1 = (H_1, h_1), \ldots, \overline{H}_r = (H_r, h_r)$ be a basis of $H$.
We set 
\[
\overline{D}_n = a_{n1}\overline{H}_1 + \cdots + a_{nr}\overline{H}_r
\]
and
$a_{i} = \lim_{n\to\infty} a_{ni}$ for $i=1, \ldots, r$.
Note that $\overline{D} = a_1\overline{H}_1 + \cdots + a_{r} \overline{H}_r$.
Let $\Gamma$ be a prime divisor on $X$. Then
\begin{align*}
\ord_{\Gamma}(D) & = a_1 \ord_{\Gamma}(H_1) + \cdots + a_r \ord_{\Gamma}(H_r) \\
& =
\lim_{n\to\infty} \{ a_{n1} \ord_{\Gamma}(H_1) + \cdots + a_{nr} \ord_{\Gamma}(H_r) \} \\
& = \lim_{n\to\infty} \ord_{\Gamma}(D_n)
\geqslant 0,
\end{align*}
so that $D \geqslant 0$.

Here
we set
\[
g_n := a_{n1} h_1 + \cdots + a_{nr} h_r + \theta_{n}\quad\text{and}\quad
g := a_{1} h_1 + \cdots + a_{r} h_r.
\]
For each $v \in M_K (= M_K^{\mathrm{fin}} \cup M_K^{\infty})$,
we need to show that $g_v \geqslant 0$ under the assumption $g_{n, v} \geqslant 0$ for all $n$.
Note that $g_v$ is continuous on $X^{\mathrm{an}}_v\setminus \Supp(D)_v^{\mathrm{an}}$ and
$g_v(x) = \infty$ for $x \in \Supp(D)_v^{\mathrm{an}}$.
We assume that the non-negativity of $g_v$ does not hold.
Then there is an open set $U$ of $X^{\mathrm{an}}_v\setminus \Supp(D)_v^{\mathrm{an}}$ such that
$g_v < 0$ on $U$.
Choose $x \in U \setminus \left(\bigcup_{i=1}^r \Supp(H_i)\right)_v^{\mathrm{an}}$. Then 
\[
g_{n, v}(x) = \theta_{n,v}(x) + \sum_{i=1}^r a_{ni} h_{i, v}(x) \geqslant 0,
\]
which implies
\begin{align*}
g_v(x) & = \sum_{i=1}^r a_{i} h_{i, v}(x) \\
& = \lim_{n\to\infty} \theta_{n,v}(x) + \sum_{i=1}^r \left( \lim_{n\to\infty} a_{ni} \right) h_{i,v}(x) \\
& = \lim_{n\to\infty}
\left( \theta_{n,v}(x) + \sum_{i=1}^r a_{ni} h_{i,v}(x) \right) \geqslant 0.
\end{align*}
This is a contradiction.
\end{proof}

\begin{Lemma}
\label{lemma:convex:compact}
Let $\overline{D}$ be an adelic arithmetic $\RR$-Cartier divisor of $C^0$-type.
Then we have the following:
\begin{enumerate}
\item
$\hat{\Gamma}(X, \overline{D})^{\times}_{\RR}$ is a convex set
(for the definition of $\hat{\Gamma}(X, \overline{D})^{\times}_{\RR}$, see
Conventions and terminology~\ref{CT:adelic:R:Cartier:div}).

\item
Let $H$ be a finite dimensional vector subspace of $\Rat(X)^{\times}_{\RR}$.
Then $H \cap \hat{\Gamma}(X, \overline{D})^{\times}_{\RR}$ is compact.
\end{enumerate}
\end{Lemma}

\begin{proof}
(1) For $t \in [0, 1]$ and $s , s'\in \hat{\Gamma}(X, \overline{D})^{\times}_{\RR}$,
\begin{align*}
\overline{D} + \widehat{(s^{t} {s'}^{(1-t)})} & = \overline{D} + t\widehat{(s)}+  (1-t) \widehat{(s')} \\
& =
t(\overline{D} + \widehat{(s)}) + (1-t)(\overline{D} + \widehat{(s')}) \geqslant 0,
\end{align*}
so that $\hat{\Gamma}(X, \overline{D})^{\times}_{\RR}$ is convex.

\medskip
(2)
We can find a model $(\XXX, \DDD)$ of $(X, D)$ and an $F_{\infty}$-invariant $D$-Green function $k$ of
$C^0$-type on $X(\CC) = \bigcup_{\sigma \in K(\CC)} X_{\sigma}^{\mathrm{a}}$ such that
$(\DDD, k)^{a} \geqslant \overline{D}$, where
$(\DDD, k)^{a}$ is the associated adelic $\RR$-Cartier divisor of $C^0$-type on $X$. 
Then 
\[
\hat{\Gamma}(X, (\DDD, k)^a)^{\times}_{\RR} \supseteq \hat{\Gamma}(X, \overline{D})^{\times}_{\RR}.
\]
By Lemma~\ref{lemma:injective:principal} together with \cite[Corollary~3.3.2]{MoD},
$H \cap \hat{\Gamma}(X, (\DDD, k)^a)^{\times}_{\RR}$ is compact.
Moreover, by Lemma~\ref{lemma:limit:effective}, we can see that $H \cap \hat{\Gamma}(X, \overline{D})^{\times}_{\RR}$
is closed, so that $H \cap \hat{\Gamma}(X, \overline{D})^{\times}_{\RR}$ is compact.
\end{proof}

\subsection{Algebraic dynamic system and a sufficient condition for the Dirichlet property}
\label{subsec:dynamic:system}

Let $f : X \to X$ be a surjective endomorphism of $X$ over $K$.
Let $D$ be an $\RR$-Cartier divisor on $X$.
We assume that there are a real number $d$ and $\varphi \in \Rat(X)^{\times}_{\RR}$ such that
$d > 1$ and $f^*(D) = dD + (\varphi)$.
An adelic arithmetic $\RR$-Cartier divisor $\overline{D} = (D, g)$ of $C^0$-type
is called the \emph{canonical compactification} of $D$ if
$f^*(\overline{D}) = d\overline{D} + \widehat{(\varphi)}$.
Note that $\overline{D}$ is uniquely determined by the equation
$f^*(\overline{D}) = d\overline{D} + \widehat{(\varphi)}$
(for details, see \cite[Section~3]{DsysDirichlet} or Proposition~\ref{prop:unique:existence:canonical:Green:function}).
By Lemma~\ref{lem:canonical:Green:principal},
for $\theta \in \Rat(X)^{\times}_{\RR}$, we have the following:
\begin{enumerate}
\renewcommand{\labelenumi}{(\roman{enumi})}
\item
$f^*(D + (\theta)) = d(D + (\theta)) + \left(f^*(\theta)\theta^{-d}\varphi\right)$.

\item
The canonical compactification of $D + (\theta)$ is given by
$\overline{D} + \widehat{(\theta)}$.
\end{enumerate}
Let $g_0 = \{ g_{0, v} \}_{v \in M_K}$ be a family of $D$-Green function of $C^0$-type 
on $X$.
We choose a collection of continuous functions $\lambda = \{ \lambda_v \}_{v \in M_K}$ such that
\begin{equation}\label{pullback of divisor}
f^*(D, g_0) = d (D, g_0) + \widehat{(\varphi)} + (0, \lambda).
\end{equation}
As in Subsection~\ref{subsec:canonical:green:function}, we set
\[
\begin{cases}
{\displaystyle g_n := g_0 + \sum_{i=0}^{n-1} \frac{1}{d^{i+1}}(f^i)^*(\lambda)} &  (n \geqslant 1)\\[1.5ex]
{\displaystyle \varphi_0 = 1, \quad \varphi_n = \prod_{i=0}^{n-1} (f^i)^*(\varphi)^{1/d^{i+1}}} &  (n \geqslant 1).
\end{cases}
\]
By Lemma~\ref{lem:uniform:conv},
\[
h_n = \sum_{i=0}^{n-1} \frac{1}{d^{i+1}}(f^i)^*(\lambda)\quad(n \geqslant 1),
\]
converges to a continuous function $h$ uniformly.
Here we set $g = g_0 + h$.
Then the pair $\overline{D} = (D, g)$ yields
the canonical compactification of $D$ (for details, see \cite[Section~3]{DsysDirichlet} or Proposition~\ref{prop:unique:existence:canonical:Green:function}).
Note that $g$ does not depend on the choice of the initial Green function $g_0$.
By Lemma~\ref{lem:recursive:relations}, we have the following:
\begin{enumerate}
\renewcommand{\labelenumi}{(\alph{enumi})}
\item $f^*(g_{n-1}) = d g_n  - \log |\varphi|^2$ and
$f^*(\varphi_{n-1}) = \varphi_n^d/\varphi$ for all $n \geqslant 1$.
In particular,
$f^*\left((D, g_{n-1}) + \widehat{(\varphi_{n-1})}\right) = d\left((D, g_{n}) + \widehat{(\varphi_n)}\right)$
for all $n \geqslant 1$.

\item
If $(D, g_0) \geqslant 0$, then $(D, g_{n}) + \widehat{(\varphi_n)} \geqslant 0$ for all $n \geqslant 0$.
\end{enumerate}
The vector subspace of $\Rat(X)^{\times}_{\RR}$ generated by $\{ \varphi_n \}_{n=1}^{\infty}$ 
is denoted by $V(\varphi)$.
We say that $V(\varphi)$ \emph{has the finiteness property} if 
$V(\varphi)$ is finitely generated as a vector space over $\mathbb R$.

\begin{Lemma}
The following are equivalent:
\begin{enumerate}
\renewcommand{\labelenumi}{\rom{(\arabic{enumi})}}
\item
$V(\varphi)$ has the finiteness property.

\item
There are $\RR$-rational functions $\phi_1, \ldots, \phi_l$ on $X$ and
$A, A_1, \ldots, A_l \in \RR^l$ such that
$\varphi = \phi^{A}$ and
$f^*(\phi_i) = \phi^{A_i}$ for $i=1, \ldots, l$ (see Conventions and terminology~\ref{CT:tensor:R}).
\end{enumerate}
\end{Lemma}

\begin{proof}
(1) $\Longrightarrow$ (2): Let $\{ \vartheta_1, \ldots, \vartheta_n \}$ be a basis of $V(\phi)$.
Clearly $\varphi = \vartheta^{B}$ for some $B \in \RR^n$.
For each $i$, we can find $c_1, \ldots, c_r \in \RR$ such that
$\vartheta_i = \varphi_{n_1}^{c_1} \cdots \varphi_{n_r}^{c_r}$. Thus, as
\[
f^*(\vartheta_i) = f^*(\varphi_{n_1})^{c_1} \cdots f^*(\varphi_{n_r})^{c_r}
= (\varphi_{n_1 + 1}^{dc_1}/\varphi^{c_1}) \cdots (\varphi_{n_r + 1}^{dc_r}/\varphi^{c_r}),
\]
we can find $B_i \in \RR^n$ such that $f^*(\vartheta_i) = \vartheta^{B_i}$.

\medskip
(2) $\Longrightarrow$ (1): 
Clearly $V(\varphi)$ is a vector subspace of the vector space generated by $\phi_1, \ldots, \phi_l$,
so that the assertion  follows.
\end{proof}

\begin{Lemma}
\label{lemma:compact}
We assume that $(D, g_0) \geqslant 0$ and 
$V(\varphi)$ has the finiteness property.
Then there is a subsequence $\{ \varphi_{n_i} \}_{i\in\mathbb N}$ of $\{ \varphi_n \}_{n\geqslant 1}$ such that
the limit of $\{ \varphi_{n_i} \}_{i\in\mathbb N}$ exists in the usual topology of $V(\varphi)$
as a finite dimensional vector space over $\RR$.
\end{Lemma}

\begin{proof}
First of all, note that there is a non-empty open set $U$ of $\Spec(O_K)$ such that
$\lambda_{\mathfrak p} = 0$ for all $\mathfrak p \in U$, where $\lambda_{\mathfrak p}$ is determined by \eqref{pullback of divisor}. Moreover,
there is a positive number $c$ such that $(h_n)_v \leqslant h_v + c$ for all $v \in M_K \setminus (U \cap M_K^{\mathrm{fin}})$.
Thus, if we set 
\[
\overline{D}' = (D, g) + \left(0, \sum_{v \in M_K \setminus (U \cap M_K^{\mathrm{fin}})} c [v]\right),
\]
then, by Lemma~\ref{lem:recursive:relations},
\[
\overline{D}' + \widehat{(\varphi_n)} \geqslant (D, g_n) + \widehat{(\varphi_n)} \geqslant 0.
\]
Thus, $\varphi_n \in \hat{\Gamma}(X, \overline{D})^{\times}_{\RR} \cap V(\varphi)$,
so that the assertion follows from Lemma~\ref{lemma:convex:compact}.
\end{proof}

\begin{Theorem}
\label{thm:Dirichlet:prop}
If $D$ is effective and $V(\varphi)$ has the finiteness property,
then the Dirichlet property holds.
\end{Theorem}

\begin{proof}
We can choose an initial family $g_0$ of $D$-Green functions with $(D, g_0) \geqslant 0$.
By Lemma~\ref{lemma:compact} together with Lemma~\ref{lem:recursive:relations}, 
there is a subsequence $\{ \varphi_{n_i} \}_{i\in\mathbb N}$ of $\{ \varphi_n \}_{n\geqslant 1}$ such that
the limit of $\{ \varphi_{n_i} \}_{i\in\mathbb N}$ exists in the usual topology of $V(\varphi)$
as a finite dimensional vector space over $\RR$. We denote the limit by $\varphi$.
Note that $(D, g_{n_i}) + \widehat{(\varphi_{n_i})} \geqslant 0$ by
Lemma~\ref{lem:recursive:relations}, so that, if we set $h_{n_i} = g_{n_i} - g$, then
then $\overline{D} + (0, h_{n_i}) + \widehat{(\varphi_{n_i})} \geqslant 0$.
Therefore,
by Lemma~\ref{lemma:limit:effective},
$\overline{D} + \widehat{(\varphi)} \geqslant 0$, as required.
\end{proof}

\begin{Remark}
Let $\theta \in \Rat(X)^{\times}_{\RR}$. If we set $D' = D + (\theta)$,
then, by Lemma~\ref{lem:canonical:Green:principal}, we have 
\[
f^*(D') = dD' +(f^*(\theta)\theta^{-d}\varphi).
\]
In order to apply Theorem~\ref{thm:Dirichlet:prop},
it is better to choose $\theta$ such that $D' \geqslant 0$ and
$f^*(\theta)\theta^{-d}\varphi$ is simple as much as possible.

For example, $X = \Proj(K[T_0, T_1])$, $D = \{ T_1 = 0 \}$ and $f$ is given by
\[
f(T_0 : T_1) = (T_0^2 : T_1^2 + c T_0^2)
\]
for some $c \in K$. This is a famous complex dynamic system.
If we set $z = T_0/T_1$, then $f^*(D) = 2D + (1 + cz^{2})$.
We do not know the finiteness property of $V(1 + cz^{2})$.
On the other hand, if we set $D' = \{ T_0 = 0 \}$,
then $f^*(D') = 2D'$. In this case, $V(1)$ is trivial, so that
by the above theorem, the Dirichlet property holds for $D'$ equipped with its canonical Green function (see \S\ref{subsec:canonical:green:function}).
\end{Remark}

\begin{Corollary}
\label{cor:Dirichlet:prop}
If $D$ is effective and $f^*(D) = dD$, then
the canonical compactification $\overline{D}$ of $D$ has the Dirichlet property.
\end{Corollary}

Finally let us consider examples.
\begin{Example}
\label{Example:Dirichlet}
We assume $X$ is the $n$-dimensional projective space over $K$, that is,
$X = \PP^n_K = \Proj(K[T_0, T_1, \ldots, T_n])$. Let $f : \PP^n_K \to \PP^n_K$ be
a surjective endomorphism over $K$. We assume that $f$ is a polynomial map, that is,
\[
f\left( \PP^n_K \setminus \{ T_0 = 0 \} \right) \subseteq \PP^n_K \setminus \{ T_0 = 0 \}.
\]
We set $z_i = T_i/T_0$ for $i=1, \ldots, n$.
Then there are $f_1, \ldots, f_n \in K[z_1, \ldots, z_n]$ such that
\[
f(1 : x_1 : \cdots : x_n) = (1 : f_1(x_1, \ldots, x_n) :  \cdots : f_n(x_1, \ldots, x_n)).
\]
We set $d = \max \{ \deg(f_1), \ldots, \deg(f_n) \}$. Then 
$f : \PP^n_K \to \PP^n_K$ is
given by
\[
f(T_0 : \cdots : T_n) = (T_0^d : F_1(T_0, \ldots, T_n) : \cdots : F_n(T_0, \ldots, T_n)),
\]
where $F_1, \ldots, F_n$ are homogeneous polynomials of degree $d$ with
\[
F_i(1, X_1, \ldots, X_n) = f_i(X_1, \ldots, X_n)
\]
for $i=1, \ldots, n$ 
and
\[
\{ (t_1, \ldots, t_n) \in \overline{K}^{n} \mid  F_1(0, t_1, \ldots, t_n) = \cdots =
F_n(0, t_1, \ldots, t_n) = 0 \} = \{ (0, \ldots, 0 )\}.
\]
We set $D = \{ T_0 = 0 \}$. Then $f^*(D) = d D$, so that the Dirichlet property holds for
the canonical compactification by Corollary~\ref{cor:Dirichlet:prop}.

For example, in the case where $f_c : \PP^1_K \to \PP^1_K$ is given by 
\[
f_c(T_0 : T_1) = (T_0^2 : T_1^2 + c T_0^2)\qquad (c \in K),
\]
it is well-known that the Julia set of $f_c$ heavily depends on the choice of $c$.
Nevertheless, the Dirichlet property holds.
\end{Example}

\begin{Example}
The Dirichlet property is very sensitive on the choice of the dynamic system. For example,
we set $K := \QQ(\sqrt{-1})$, $X := \PP^1_{K} = \Proj(K[T_0, T_1])$ and $z := T_1/T_0$.
Let us consider two endomorphisms $f$ and $f'$ on $X$ given by
\[
f(T_0 : T_1) = (2T_0T_1 : T_1^2 - T_0^2)
\quad\text{and}\quad
f'(T_0 : T_1) = (2\sqrt{-1}T_0T_1 : T_1^2 - T_0^2),
\]
that is, $f(z) = (1/2)(z - 1/z)$ and $f'(z) = (1/2\sqrt{-1})(z - 1/z)$.
If we set $D := \{ T_1 - \sqrt{-1}T_0 = 0 \}$, then
$f^*(D) = 2D$ because 
\[
(T_1^2 - T_0^2) - \sqrt{-1} (2 T_0 T_1) = (T_1 - \sqrt{-1}T_0)^2.
\]
Let $g$ be the canonical $D$-Green function of $C^0$-type with respect to $f$.
Then, by Corollary~\ref{cor:Dirichlet:prop}, $\overline{D} = (D, g)$ has the Dirichlet property.
On the other hand, for $\sigma \in M_{K}^{\infty}$,
it is well-known that the Julia set of $f'$ on $X_{\sigma}$ is equal to $X_{\sigma}$ itself
(cf. \cite[Theorem~4.2.18]{HolDyn}).
Therefore, by \cite[Theorem~4.5]{DsysDirichlet},
for any ample $\RR$-Cartier divisor $A$,
the canonical compactification $\overline{A}$ with respect to $f'$ does not have
the Dirichlet property.
\end{Example}

\subsection{A remark on a sufficient condition for the Dirichlet property}
\label{subsec:rem:Dirichlet}
In this subsection, we do not suppose given the endomorphism $f : X \to X$.
Let $\overline{D}$ be an adelic arithmetic $\RR$-Cartier divisor of $C^0$-type on $X$.
We assume that $D$ is big and $\overline{D}$ is pseudo-effective. 
Let $\zeta$ be an adelic $\RR$-divisor on $\Spec(K)$
with $\adeg(\zeta) = 1$. Then $\overline{D} + t\pi^*(\zeta)$ is big
for all $t \in (0, \infty)$,
where $\pi :  X \to \Spec(K)$ is the canonical morphism (see \cite[\S6.2]{DsysDirichlet}).

\begin{Proposition}
Let $H$ be a finite-dimensional vector subspace of $\Rat(X)^{\times}_{\RR}$.
If there is a sequence $\{ t_n \}_{n\in\mathbb N}$ of positive numbers such that $\lim_{n\to\infty} t_n = 0$ and,
for each $n\in\mathbb N$, we can find $\theta_n \in H$ with
$\overline{D} + t_n \pi^*(\zeta) + \widehat{(\theta_n)} \geqslant 0$,
then $\overline{D}$ has the Dirichlet property.
\end{Proposition}

\begin{proof}
Let $\vartheta \in K^{\times}_{\RR}$ such that
$\zeta + \widehat{(\vartheta)} \geqslant 0$.
Then 
\[
\overline{D} + t_n \pi^*(\zeta) + \widehat{(\theta_n)} =
\overline{D} + t_n (\pi^*(\zeta + \widehat{(\vartheta)}) + \widehat{(\theta_n\pi^*(\vartheta)^{-t_n})}, 
\]
so that, replacing $H$ by the vector subspace generated by $H$ and $\pi^*(\vartheta)$,
we may assume that $\zeta \geqslant 0$.

We choose $t > 0$ such that $t_n \leqslant t$ for all $n$.
Then 
\[
\theta_n \in H \cap \hat{\Gamma}(X, \overline{D} + t\pi^*(\zeta))^{\times}_{\RR}
\]
for all $n$. Note that $H \cap
\hat{\Gamma}(X, \overline{D} + t\pi^*(\zeta))^{\times}_{\RR}$ is a compact convex set
by Lemma~\ref{lemma:convex:compact}.
Therefore, there is a subsequence $\{ \theta_{n_k} \}_{k\in\mathbb N}$ of $\{ \theta_n \}_{n\in\mathbb N}$ such that
the limit $\theta$ of $\{ \theta_{n_k} \}_{k\in\mathbb N}$ exists.
Moreover, by Lemma~\ref{lemma:limit:effective}, we have
$\overline{D} + \widehat{(\theta)} \geqslant 0$.
\end{proof}

\begin{Example}[Toric variety]
Let $N$ be a free $\ZZ$-module of rank $n$ and $M = \Hom_{\ZZ}(N, \ZZ)$.
Let $\Sigma$ be a complete fan in $N_{\RR}$.
Let $X = X(\Sigma)_K$ be the toric variety over $K$ associated with $\Sigma$.
Let $D$ be a big toric $\RR$-Cartier divisor on $X$ and
$g$ a family of $D$-Green functions of toric type.
If $\overline{D}$ is big,
then we can find $m \in M_{\RR}$ such that
$\overline{D} + \widehat{(\chi^{m})} \geqslant 0$.
Therefore, we can see that any pseudo-effective adelic arithmetic $\RR$-Cartier divisor of toric type
has the Dirichlet property.
\end{Example}

\section{Arakelov geometry over a trivially valued field}

In this section, we introduce the analogue of Arakelov geometry in the setting of projective varieties over a trivially valued field. Throughout the section, let $K$ be a field and $|\ndot|$ be the trivial absolute value on $K$. Namely $|a|=1$ if $a\in K^{\times}$ and $|0|=0$.

\begin{Definition}\label{Def:principal}
Let $X$ be an integral projective scheme over $\Spec K$. 
By \emph{adelic $\mathbb R$-Cartier divisor} on $X$, we refer to a couple $\overline D=(D,g)$, where $D$ is an $\mathbb R$-Cartier divisor on $X$ and $g$ is a $D$-Green function of $C^0$-type. We say that an adelic $\mathbb R$-Cartier divisor $(D,g)$ is \emph{effective} if $D$ is effective 
as an $\RR$-Cartier divisor
and $g$, viewed as a map from $X^{\mathrm{an}}$ to $\mathbb R\cup\{+\infty\}$ (see Remark \ref{Rem:effective}), is a non-negative function.

The set $\widehat{\mathrm{Div}}_{\mathbb R}(X)$ of adelic $\mathbb R$-Cartier divisors on $X$ forms a vector space over $\mathbb R$. The map from $\mathrm{Rat}(X)_{\mathbb R}^{\times}$ to $\widehat{\mathrm{Div}}_{\mathbb R}(X)$ sending $f$ to $\widehat{(f)}:=((f),-\log|f|)$ is $\mathbb R$-linear. The adelic $\mathbb R$-Cartier divisors lying in the image of this map are said to be \emph{principal}. If two adelic $\mathbb R$-Cartier divisors differ by a principal one, then we say that they are \emph{$\mathbb R$-linearly equivalent}.
\end{Definition}

\subsection{Global section space and sup norm}
Let $X$ be an integral projective scheme over $\Spec K$ and $(D,g)$ be an adelic $\mathbb R$-Cartier divisor on $X$. Let 
\[H^0(D):=\{s\in\mathrm{Rat}(X)^{\times}\,:\,(s)+D \geqslant_{\RR} 0
\}\cup\{0\}\subset\mathrm{Rat}(X).\]
This is a finite dimensional $K$-vector subspace of $\mathrm{Rat}(X)$ (cf. Remark~\ref{rem:finite:dim:H0}).

Let $s$ be a non-zero
element in $H^0(D)$. As $(s)+D$ is an effective $\mathbb R$-Cartier divisor on $X$ and $g-\log|s|$ is a Green function of $(s)+D$, we obtain that 
$|s|\mathrm{e}^{-g} = \mathrm{e}^{-g  + \log |s|}$, which is denoted by $|s|_g$, 
is continuous on $X^{\mathrm{an}}$ by Proposition~\ref{Pro:e-gextension}. We define
$\| s \|_g$ to be
\[
\| s \|_g := \sup \{ |s|_g(x) \mid x \in X^{\mathrm{an}} \}.
\]
Then $\|\ndot\|_g:H^0(D)\rightarrow\mathbb R_+$ actually yields an ultrametric norm on the $K$-vector space $H^0(D)$. Note that $\{ \| s \|_g \mid s \in H^0(D) \}$ is a finite set because
the absolute value of $K$ is trivial (cf. \cite[\S1.2.1]{Extension}).
The function $\|\ndot\|_g$
defines a decreasing $\mathbb R$-filtration on the vector space $H^0(D)$ as follows:
\[\forall\,t\in\mathbb R,\quad \mathcal F^t(H^0(D)):=
\{s\in H^0(D)\,:\, \| s\|_g \leqslant e^{-t}\},\]
which is a vector subspace of $H^0(D)$ because the absolute value of $K$ is trivial.
The function $t\mapsto \mathrm{rk}_{K}(\mathcal F^t(H^0(D)))$ is decreasing and left-continuous. Moreover, $\mathcal F^t(H^0(D))$ reduces to the zero vector space when $t$ is sufficiently positive, and $\mathcal F^t(H^0(D))=H^0(D)$ when $t$ is sufficiently negative.  
We let
\[
\widehat{\deg}_+(D,g):={\int_{0}^{+\infty}\mathrm{rk}_K(\mathcal F^t(H^0(D)))}\,\mathrm{d}t\]
and
\[\lambda_{\max}(D,g):=\sup\{t\in\mathbb R\,:\,\mathcal F^t(H^0(D))\neq\{0\}\}.\]
We refer the readers to \cite[\S 1.2]{Chen10} for more details on $\mathbb R$-filtered vector spaces and to \cite[\S3.2]{Chen10b} for the comparison of the invariant $\widehat{\deg}_+$ and the logarithm of the number of small sections in the classic setting of arithmetic geometry over a number field.

\begin{Remark}
\label{rem:finite:dim:H0}
The finite-dimensionality of
$H^0(X, D)$ can be checked in the following way:

\medskip
{\bf Step 1}: 
Clearly we may assume that $H^0(X, D) \not= \{ 0 \}$, so that we can choose
$s \in H^0(X, D) \setminus \{ 0 \}$.
Then a homomorphism $H^0(X, D) \to H^0(X, D + (s))$ given by $\phi \mapsto \phi s^{-1}$
yields an isomorphism. Therefore, replacing $D$ by $D + (s)$, we may suppose that
$D \geqslant_{\RR} 0$.

\medskip
{\bf Step 2}: By Step~1, there are effective Cartier divisors $D_1, \ldots, D_r$ and
$a_1, \ldots, a_r \in \RR_{>0}$ such that $D = a_1 D_1 + \cdots + a_r D_r$.
We choose integers $a'_1, \ldots, a'_r$ such that $a_i \leqslant a'_i$ for all $i$.
If we set $D' = a'_1 D_1 + \cdots + a'_r D_r$, then $H^0(X, D) \subseteq H^0(X, D')$,
so that we may assume that $D$ is a Cartier divisor.

\medskip
{\bf Step 3}: Let $\mu : X' \to X$ be the normalization of $X$. Then $H^0(X, D) \subseteq H^0(X', \mu^*(D))$, and
hence we may assume that $X$ is normal. Therefore, by using Hartogs' property,
we can see that the natural homomorphism $H^0(X, \OO_X(D)) \to H^0(X, D)$ is bijective, as required.
\end{Remark}

\subsection{Height and essential minimum}
\label{subsec:Height:essential:minimum}
If $(D,g)$ is an adelic $\mathbb R$-Cartier divisor on $X$, we let
\[\lambda_{\max}^{\mathrm{asy}}(D,g):=\limsup_{n\rightarrow+\infty}\frac{1}{n}{\lambda_{\max}(nD,ng)}.\]
Since the sequence $\{\lambda_{\max}(nD,ng)\}_{n\geqslant 1}$ is super-additive, we obtain that 
\[\lambda_{\max}^{\mathrm{asy}}(D,g)=\sup_{n\geqslant 1}\frac{1}{n}{\lambda_{\max}(nD,ng)}.\]
This invariant is closely related to the analogue in the setting of arithmetic geometry over a trivially valued field of the essential minimum of height function.

Here let us introduce the height function $h^{\mathrm{an}}_{(D,g)}$ on $X^{\mathrm{an}}$ associated with $(D,g)$.
Fix a point $\xi$ of $X^{\mathrm{an}}$. 
Let $p_{\xi} \in X$ be the associated scheme point of $\xi$ and $\kappa(\xi)$
be the residue field of $p_{\xi}$.
The point $\xi$ gives rise to an absolute value $v_{\xi}$ on $\kappa(\xi)$.
Note that $v_{\xi}$ is non-archimedean because $v_{\xi}$ is trivial on $K$.
We set 
\[
\mathfrak o_{\xi} := \{ \alpha \in \kappa(\xi) \mid v_{\xi}(\alpha) \leqslant 1 \}
\quad\text{and}\quad
\mathfrak m_{\xi} := \{ \alpha \in \kappa(\xi) \mid v_{\xi}(\alpha) < 1 \}.
\]
In the case where $v_{\xi}$ is trivial, $\mathfrak o_{\xi} = \kappa(\xi)$ and
$\mathfrak m_{\xi} = \{ 0 \}$.
Since $X$ is proper over $\Spec K$,
by the valuative criterion of properness there is a unique $K$-morphism
$\Spec(\mathfrak o_{\xi}) \to X$ such that the following diagram is commutative:\
\[
\begin{tikzcd}
\Spec(\mathfrak o_{\xi}) \arrow[r] & X \\
\Spec(\kappa(\xi)) \arrow[u]\arrow[ru] &
\end{tikzcd}
\]
where $\Spec(\kappa(\xi)) \to \Spec(\mathfrak o_{\xi})$ and
$\Spec(\kappa(\xi)) \to X$ are the canonical morphisms.
The image of $\mathfrak m_{\xi}$ by $\Spec(\mathfrak o_{\xi}) \to X$ is
denoted by $r_{\xi}$, which is called the \emph{reduction point} of $\xi$.

Let $f$ be a local equation of $D$ on a Zariski open set $U$ containing $r_{\xi}$.
Note that $\xi \in U^{\an}$ because $p_{\xi} \in U$.
By definition the function $g+\log|f|$ extends to a continuous function 
$\vartheta_f$ on $U^{\mathrm{an}}$. 
Here we consider the evaluation $\vartheta_f(\xi)$ of $\vartheta_f$ at $\xi$.
It does not depend on the choice of $U$ and $f$.
Indeed, let $f'$ be another local equation of $D$ on a Zariski open set $U'$ containing $r_{\xi}$. 
Then there is $u \in (\OO_{X,r_x}^{\times})_{\RR}$ with
$f' = u f$, so that the extension $\vartheta_{f'}$ of $g + \log |f'|$ is equal to
$\vartheta_f + \log |u|$ around $\xi$, and hence the assertion follows because $|u|(\xi) = 1$. 
Thereore it is denoted by  $h^{\mathrm{an}}_{(D,g)}(\xi)$.
For any point $x$ of $X$, we denote by $x^{\mathrm{an}}$ the point in $X^{\mathrm{an}}$ corresponding to the point $x$ and the trivial absolute value on the residue field of $x$. 
We define $h_{(D, g)}(x)$ to be
$h_{(D, g)}(x) := h^{\mathrm{an}}_{(D, g)}(x^{\mathrm{an}})$.

For a Cartier divisor $E$ on $X$, we say $E$ is {\em semiample} if $\OO_X(mE)$ is generated by global sections
for some positive integer $m$.
In general, an $\RR$-Cartier divisor $D$ on $X$ is said to be {\em semiample} if there are
semiample Cartier divisors $E_1, \ldots, E_r$ on $X$ and $a_1, \ldots, a_r \in \RR_{>0}$ with
$D = a_1 E_1 + \cdots + a_r E_r$.
The following proposition contains basic properties of the height functions.

\begin{Proposition}
\label{prop:basic:prop:height}
Let $(D, g)$ and $(D', g')$ be
adelic $\RR$-Cartier divisors on $X$. Then we have the following:
\begin{enumerate}
\renewcommand{\labelenumi}{\rom{(\arabic{enumi})}}
\item
$h^{\mathrm{an}}_{a(D, g) + a'(D', g')}(\xi) = a h^{\mathrm{an}}_{(D, g)}(\xi) + a' h^{\mathrm{an}}_{(D', g')}(\xi)$ for all 
$\xi \in X^{\mathrm{an}}$ and $a, a' \in \RR$.

\item
$h^{\mathrm{an}}_{\widehat{(s)}}(\xi) = 0$ for all $\xi \in X^{\mathrm{an}}$ and $s \in \Rat(X)^{\times}_{\RR}$.

\item
Let $\pi : Y \to X$ be a morphism of integral projective schemes over $K$ such that
$\pi(Y) \not\subseteq \Supp(D)$. Then $h^{\mathrm{an}}_{\pi^*(D, g)}(\zeta) = h^{\mathrm{an}}_{(D, g)}(\pi^{\mathrm{an}}(\zeta))$
for all $\zeta \in Y^{\mathrm{an}}$.

\item
If $D$ is semiample, then there is a constant $C$ such that $h^{\mathrm{an}}_{(D, g)}(\xi) \geqslant C$ for all
$\xi \in X^{\mathrm{an}}$.
\end{enumerate}
\end{Proposition}

\begin{proof}
In the following proof, for $\xi \in X^{\mathrm{an}}$,
let $f$ and $f'$ be local equations of $D$ and $D'$ over an Zariski open set $U$ containing $r_{\xi}$, respectively.
Let $\vartheta_{f}$ and $\vartheta_{f'}$ be the continuous extensions of
$g + \log |f|$ and $g' + \log |f'|$ over $U^{\mathrm{an}}$, respectively.

(1) Note that $f^{a} {f'}^{a'}$ yields a local equation of $aD + a'D'$, and that
the continuous extension of 
\[
(a g + a'g') + \log |f^a {f'}^{a'}| = a (g + \log |f|) + a'(g' + \log |f'|)
\]
is $a \vartheta_f + a' \vartheta_{f'}$.
Therefore,
\[
h^{\mathrm{an}}_{a(D, g) + a'(D', g')}(\xi) = (a \vartheta_f + a' \vartheta_{f'})(\xi) = 
a h^{\mathrm{an}}_{(D, g)}(\xi) + a' h^{\mathrm{an}}_{(D', g')}(\xi).
\]

\medskip
(2) A local equation of $(s)$ is given by $s$, so that the continuous extension of
$- \log |s| + \log |s|$ is zero, as required.

\medskip
(3) For $\zeta \in Y^{\mathrm{an}}$, we set $\xi = \pi^{\mathrm{an}}(\zeta)$.
As $\pi^*(f)$ is a local equation of $\pi^*(D)$ over $\pi^{-1}(U)$,
the continuous extension of ${\pi^{\mathrm{an}}}^*(g) + \log |\pi^*(f)|$ is ${\pi^{\mathrm{an}}}^*(\vartheta_f)$.
Therefore,
\[
h^{\mathrm{an}}_{\pi^*(D, g)}(\zeta) = {\pi^{\mathrm{an}}}^*(\vartheta_f)(\zeta) 
= \vartheta_f(\pi^{\mathrm{an}}(\zeta)) = h^{\mathrm{an}}_{(D, g)}(\pi^{\mathrm{an}}(\zeta)).
\]

\medskip
(4) First we assume that $D$ is a Cartier divisor and $\OO_X(mD)$ is generated by global sections
for a positive integer $m$. Let $\{ s_0, \ldots, s_N \}$ be a basis  of $H^0(\OO_X(mD))$.
We consider a morphism $\pi : X \to \PP_K^N = \Proj(K[T_0, \ldots, T_N])$ given by
$x \mapsto (s_0(x) : \cdots : s_N(x))$. We set $H_0 := \{ T_0 = 0 \}$,
$z_{ij} := T_i/T_j$ ($0 \leqslant i, j \leqslant N$)
and $h_0 := \log \max \{ 1, |z_{10}|, \ldots, |z_{N0}| \}$.
Note that $h_0$ is a Green function of $H_0$.  

Here let us see that $h^{\mathrm{an}}_{(H_0, h_0)}(\zeta) \geqslant 0$ for all $\zeta \in \PP^{N, \mathrm{an}}_K$.
We assume that $r_{\zeta} \in U_{i} = \{ T_i \not= 0 \}$.
A local equation of $H_0$ on $U_i$ is given by $z_{0i}$  and
the continuous extension of $h_0 + \log |z_{0i}|$ is
\[
\log \max \{ |z_{0i}|, |z_{1i}|, \ldots, |z_{i-1 i}|, 1, |z_{i+1 i}|, \ldots, |z_{Ni}|  \},
\]
so that the assertion follows because the above function is non-negative on $U_i^{\mathrm{an}}$.

There is $s \in \Rat(X)^{\times}$ such that $\pi^*(H_0) = m D + (s)$, so that
we can find a continuous function $\theta$ on $X^{\mathrm{an}}$ such that
${\pi^{\mathrm{an}}}^*(h_0) = mg - \log |s| + \theta$. Since $\theta$ is a continuous function on the compact space
$X^{\mathrm{an}}$, there is a constant $C'$ such that $\theta \leqslant C'$ on $X^{\mathrm{an}}$.
Thus, for $\xi \in X^{\mathrm{an}}$, by using (1), (2) and (3) together with
the non-negativity of $h^{\mathrm{an}}_{(H_0, h_0)}$,
\begin{align*}
0 & \leqslant h^{\mathrm{an}}_{(H_0, h_0)}({\pi^{\mathrm{an}}}(\xi)) = h^{\mathrm{an}}_{\pi^*(H_0, h_0)}(\xi) = h^{\mathrm{an}}_{(mD + (s), mg - \log |s| + \theta)}(\xi) \\
& = h^{\mathrm{an}}_{m(D,g) + \widehat{(s)} + (0, \theta)}(\xi)
= m h^{\mathrm{an}}_{(D,g)}(\xi) + h^{\mathrm{an}}_{\widehat{(s)}}(\xi) + h^{\mathrm{an}}_{(0, \theta)}(\xi) \\
& = m h^{\mathrm{an}}_{(D,g)}(\xi) + \theta(\xi) \leqslant m h^{\mathrm{an}}_{(D,g)}(\xi) + C',
\end{align*}  
so that $h_{(D,g)}(\xi) \geqslant -C'/m$, as required.

\smallskip
Next we consider the general case, that is,
there are semiample Cartier divisors $E_1,\ldots,E_r$ on $X$ and $a_1, \ldots, a_r \in \RR_{>0}$
such that $D = a_1 E_1 + \cdots + a_r E_r$. We can find Green functions $e_1, \ldots, e_r$ of $E_1,
\ldots, E_r$, respectively such that $g = a_1 e_1 + \cdots + a_r e_r$.
By the previous observation, for each $i=1, \ldots, r$,
there is a constant $C_i$ such that $h^{\mathrm{an}}_{(E_i, e_i)} \geqslant C_i$
on $X^{\mathrm{an}}$. Therefore, by (1),
\[
h^{\mathrm{an}}_{(D, g)} = a_1 h^{\mathrm{an}}_{(E_1, e_1)} + \cdots + a_r h^{\mathrm{an}}_{(E_r, e_r)} \geqslant a_1 C_1 + \cdots + a_r C_r
\] 
on $X^{\mathrm{an}}$.                     
\end{proof}

We define the \emph{essential minimum} of $(D,g)$ as
\[\widehat{\mu}_{\mathrm{ess}}(D,g):=\sup_{Z\subsetneq X}\inf_{\begin{subarray}{c}x\in X\setminus Z\\ x\text{ closed}
\end{subarray}}h_{(D,g)}(x),\]
where $Z$ runs over the set of strict closed subschemes of $X$.

\begin{Proposition}\label{Pro:ess min g}
Let $(D,g)$ be an adelic $\mathbb R$-Cartier divisor on $X$. One has $\widehat{\mu}_{\mathrm{ess}}(D,g)= g(\eta_0)$, where $\eta_0$ denotes the point in $X^{\mathrm{an}}$ corresponding to the generic point of $X$ and the trivial absolute value on $\mathrm{Rat}(X)$.
\end{Proposition}
\begin{proof}
Let $\alpha$ be a real number, $\alpha>g(\eta_0)$. We consider $g$ as a continuous function on certain $U^{\mathrm{an}}$, where $U$ is a non-empty Zariski open subset of $X$. The set $\{x\in U^{\mathrm{an}}\,:\,g(x)<\alpha\}$ is an open subset of $U^{\mathrm{an}}$ (for the Berkovich topology), which contains the point $\eta_0$. Thus there is a non-empty Zariski open subset $V\subset U$, an invertible regular function $f$ on $V$ and an open subset $A$ of $\mathbb R$ such that 
\[\eta_0\in |f|^{-1}(A)\subset \{x\in U^{\mathrm{an}}\,:\,g(x)<\alpha\}.\]
Note that $|f|(\eta_0)=1$. Moreover, since $f$ is invertible, for any closed point $x\in V$, one has $|f|^{-1}(x)=1$, where we have identified $x$ with the point in $X^{\mathrm{an}}$ corresponding to $x$ and the trivial absolute value on the residue field $\kappa(x)$. Therefore all closed point in $V$ are contained in $\{x\in U^{\mathrm{an}}\,:\,g(x)<\alpha\}$. In other words, the set of closed points of height $\leqslant\alpha$ is dense in $X$. Hence $\widehat{\mu}_{\mathrm{ess}}(D,g)$ is bounded from above by $g(\eta_0)$. 

Conversely, if $\beta$ is a real number such that $\beta<g(\eta_0)$, then $\{x\in U^{\mathrm{an}}\,:\,g(x)>\beta\}$ is also an open subset of $U^{\mathrm{an}}$ which contains the point $\eta_0$. By the same method as above, we obtain the existence of a non-empty Zariski open subset $V\subset U$ such that any closed point $x\in V$ satisfies $g(x)>\beta$. In other words, the set of closed points $y\in X$ such that $g(y)\leqslant\beta$ is contained in the Zariski closed subset $X\setminus V$. Since $\beta$ is arbitrary, we obtain that $\widehat{\mu}_{\mathrm{ess}}(D,g)$ is bounded from below by $g(\eta_0)$. 
\end{proof}

\begin{Remark}
Let $\{x_n\}_{n\in\mathbb N}$ be a sequence of closed points in $X$ which is generic (namely every subsequence of $\{x_n\}_{n\in\mathbb N}$ is Zariski dense in $X$) and such that \[\lim_{n\rightarrow+\infty}h_{(D,g)}(x_n)=\widehat{\mu}_{\mathrm{ess}}(D,g)\]
for certain adelic $\mathbb R$-Cartier divisor $(D,g)$ with $D$ big. For any $n\in\mathbb N$, let $\mu_n$ be the Borel probability measure on $X^{\mathrm{an}}$ defined as the distribution of the average on the Galois orbite of $x_n^{\mathrm{an}}$ (under the action of $\mathrm{Gal}(\overline K/K)$). Then the sequence $\{\mu_n\}_{n\in\mathbb N}$ converges weakly to the Dirac measure on $\eta_0$. This assertion can be deduced from the fact that $\widehat{\mu}_{\mathrm{ess}}$ is a linear form on the vector space of adelic $\mathbb R$-Cartier divisors, by using the technics in \cite[\S5.2]{Chen11a}. Compared to classic equidistribution results in Arakelov geometry such as \cite{Szpiro_Ullmo_Zhang}, or the $p$-adic analogue proved by Chamber-Loir \cite{Chambert-Loir06} (see also \cite{Chambert-Loir10} for a survey on the related problems), the above equidistribution result does not require neither the equality between the essential minimum of the height function and the normalised Arakelov height of the variety, nor any positivity condition on the Green function.
\end{Remark}

\begin{Corollary}
\label{cor:finite:asy:max}
Let $(D,g)$ be an adelic $\mathbb R$-Cartier divisor on $X$. One has $\lambda_{\max}^{\mathrm{asy}}(D,g)\leqslant\widehat{\mu}_{\mathrm{ess}}(D,g)$. In particular, $\lambda_{\max}^{\mathrm{asy}}(D,g)<+\infty$.
\end{Corollary}
\begin{proof}
Let $n$ be an integer, $n\geqslant 1$. If $s$ is a non-zero element in $H^0(nD)$, then $ng-\log|s|$ defines a continuous function on $X^{\mathrm{an}}$ valued in $\mathbb R\cup\{+\infty\}$. Therefore, for any closed point $x$ of $X$ outside of the support of $(s)$, one has  
$-\log \| s\|_{ng} \leqslant ng(x)$.
Therefore 
$-\log \| s\|_{ng} \leqslant n\widehat{\mu}_{\mathrm{ess}}(D,g)$. The second assertion follows from Proposition \ref{Pro:ess min g}.
\end{proof}

\begin{Remark}
For a subset $S$ of $X^{\mathrm{an}}$, we define $\Supp_{\mathrm{ess}}(S)$ to be
\[
\Supp_{\mathrm{ess}}(S) := \bigcap_{Z \subsetneq X} 
\overline{\{ \xi \in S \mid r_{\xi} \not\in Z \}},
\]
where $Z$ runs over all strict closed subschemes of $X$.
Here we consider 
\[
X_{\leqslant 0}^{\mathrm{an}} := \{ \xi \in X^{\mathrm{an}} \mid h^{\mathrm{an}}_{(D,g)}(\xi) \leqslant 0 \}
\]
as a subset of $X^{\mathrm{an}}$.
If $(D, g) + \widehat{(s)} \geqslant 0$ for some $s \in \Rat(X)^{\times}_{\RR}$,
then
\[
\Supp_{\mathrm{ess}}\left(X_{\leqslant 0}^{\mathrm{an}}\right) \cap \{ \xi \in X^{\mathrm{an}} \mid |s|_{g}(\xi) < 1 \} = \emptyset.
\]
This can be proved in the similar way as \cite[Lemma~2.1]{DsysDirichlet}.
Indeed, we set $Y := \Supp(D + (s))$. It is sufficient to see that
\[
\left\{ \xi \in X_{\leqslant 0}^{\mathrm{an}} \mid r_{\xi} \not\in Y \right\} \subseteq  \{ \xi \in X^{\mathrm{an}} \mid |s|_{g}(\xi)\geqslant 1 \}
\]
because $\{ \xi \in X^{\mathrm{an}} \mid |s|_{g}(\xi) \geqslant 1 \}$ is closed. For $\xi \in X_{\leqslant 0}^{\mathrm{an}}$ with $r_{\xi} \not\in Y$,
we choose a Zariski open set $U$ containing $r_{\xi}$ and a local equation $f$ of $D$ over $U$. 
As
$(g - \log |s|) + \log |f s| = g + \log |f|$ and $|f s|(\xi) = 1$ (because $fs \in (\OO_{X,r_{\xi}}^{\times})_{\RR}$), we have
\[
(g - \log |s|)(\xi) = h^{\mathrm{an}}_{(D, g) + \widehat{(s)}}(\xi) = h^{\mathrm{an}}_{(D, g)}(\xi) \leqslant 0,
\]
which means that $|s|_{g}(\xi) \geqslant 1$, as required. In particular, we have
\[
\bigcap_{Z \subsetneq X} 
\overline{\{ x^{\an} \mid \text{$x \in X \setminus Z$ and $h_{(D,g)}(x) \leqslant 0$}\}}
\cap \{ \xi \in X^{\mathrm{an}} \mid |s|_{g}(\xi) < 1 \} = \emptyset
\]
because $r_{x^{\mathrm{an}}} = x$ for $x \in X$.

\end{Remark}

\begin{Proposition}
\label{prop:lambda:mu:rational:equiv}
Let $X$ be an integral projective scheme over $\Spec K$ and $(D,g)$ be an adelic $\mathbb R$-Cartier divisor on $X$.  For any  $s \in \Rat(X)^{\times}$, we have the following:
\[
\begin{cases}
\lambda_{\max}(D, g) = \lambda_{\max}(D + (s), g - \log |s|),\\
\lambda_{\max}^{\mathrm{asy}}(D, g) = \lambda_{\max}^{\mathrm{asy}}(D + (s), g - \log |s|), \\
\hat{\mu}_{\mathrm{ess}}(D, g) = \hat{\mu}_{\mathrm{ess}}(D + (s), g - \log |s|).
\end{cases}
\]
\end{Proposition}

\begin{proof}
Note that the isomorphism $H^0(D) \to H^0(D + (s))$ given by $f \mapsto fs^{-1}$
gives rise to an isometry with respect to the norms $\|\ndot\|_g$ and $\|\ndot\|_{g - \log |s|}$,
so that the first assertion follows. The second equation is a consequence of the first one.
The last assertion results from the equality $h_{(D,g)} = h_{(D + (s), g - \log |s|)}$.
\end{proof}

\subsection{Criterion of bigness}
\label{subsec:criterion:bigness}

Let $(D,g)$ be an adelic $\mathbb R$-Cartier divisor on $X$. 
First, let us introduce the volume, the bigness and the pseudo-effectivity of $(D,g)$.
\begin{Definition}
Let $(D,g)$ be an adelic $\mathbb R$-Cartier divisor on $X$. We define the \emph{volume} of $(D,g)$ as
\[\widehat{\vol}(D,g):=\limsup_{n\rightarrow+\infty}\frac{\widehat{\deg}_+(nD,ng)}{n^{d+1}/(d+1)!},\]
where $d$ is the dimension of $X$.
If this number is positive, we say that $(D,g)$ is \emph{big}. 
An adelic $\mathbb R$-Cartier divisor $(D',g')$ is said to be \emph{pseudo-effective} if for any big adelic $\mathbb R$-Cartier divisor $(D,g)$, the sum $(D+D',g+g')$ is a big adelic $\mathbb R$-Cartier divisor. 
\end{Definition}
By definition one has
\[\widehat{\deg}_+(D,g)\leqslant\max(\lambda_{\max}(D,g),0)\mathrm{rk}_K(H^0(D)).\]
Therefore, one has
\begin{equation}\label{Equ:upper bound volume}\widehat{\mathrm{vol}}(D,g)\leqslant (d+1)\max(\lambda_{\max}^{\mathrm{asy}},0)\mathrm{vol}(D).\end{equation}
In particular, $\avol(D, g) < \infty$ by Corollary~\ref{cor:finite:asy:max}. Moreover, if $(D,g)$ is big, then $\vol(D)>0$, namely $D$ is big.

\begin{Proposition}\label{Pro: bigness and lambda}
Let $(D,g)$ be an adelic $\mathbb R$-Cartier divisor. If $(D,g)$ is big, then $\lambda_{\max}^{\mathrm{asy}}(D,g)>0$. The converse is true when $D$ is big. 
In particular, the following conditions are equivalent:
\begin{enumerate}
\renewcommand{\labelenumi}{\rom{(\alph{enumi})}}
\item
$(D, g)$ is big.

\item
$D$ is big and there are a positive integer $n_0$ and a non-zero $s \in H^0(n_0D)$ with
$\| s \|_{n_0g} < 1$.
\end{enumerate}
\end{Proposition}
\begin{proof}
By the inequality \eqref{Equ:upper bound volume}, we obtain that, if $(D,g)$ is big, then $\lambda_{\max}^{\mathrm{asy}}(D,g)$ is positive. In the following we prove that, if $D$ is big and $\lambda_{\max}^{\mathrm{asy}}(D,g)>0$, then one has $\widehat{\vol}(D,g)>0$. Let $V_\sbullet$ be the graded linear series $\bigoplus_{n\in\mathbb N}H^0(nD)$. Since $D$ is big, it contains an ample series (see \cite[Definition 1.1]{BC11}). It is moreover $\mathbb R$-filtered. By \cite[Lemma 1.6]{BC11}, for any $t\in\mathbb R$ such that $0\leqslant t<\lambda_{\max}^{\mathrm{asy}}(D,g)$, the graded linear series $V_{\sbullet}^t:=\bigoplus_{n\in\mathbb N}\mathcal F^{nt}(H^0(nD))$ contains an amples series and hence has a positive volume. Moreover, by \cite[Corollary 1.13]{BC11}, one has
\[\widehat{\vol}(D,g)=(d+1)\int_0^{\lambda_{\max}^{\mathrm{asy}}(D,g)}\vol(V_\sbullet^t)\,\mathrm{d}t>0.\]
The proposition is thus proved.
\end{proof}

\begin{Corollary}\label{Cor: pseudo-effectivity}
Let $(D,g)$ be an adelic $\mathbb R$-Cartier divisor on $X$. We assume that $D$ is big. Then $(D,g)$ is pseudo-effective if and only if $\lambda_{\max}^{\mathrm{asy}}(D,g)\geqslant 0$.
\end{Corollary}
\begin{proof}
Assume that $\lambda_{\max}^{\mathrm{asy}}(D,g)\geqslant 0$. Let $(D',g')$ be a big adelic $\mathbb R$-Cartier divisor on $X$. Since $(D',g')$ is big, by Proposition \ref{Pro: bigness and lambda} one has $\lambda_{\max}^{\mathrm{asy}}(D',g')>0$. Therefore
\[\lambda_{\max}^{\mathrm{asy}}(D+D',g+g')\geqslant\lambda_{\max}^{\mathrm{asy}}(D,g)+\lambda_{\max}^{\mathrm{asy}}(D',g')>0.\]
Since $D+D'$ is big, still by Proposition \ref{Pro: bigness and lambda} we obtain that $(D+D',g+g')$ is big.
\end{proof}

\begin{Example}
\label{exam:projective:space:big}
We assume $X = \PP^d_K = \Proj(K[T_0, \ldots, T_d])$. We set $z_i = T_i/T_0$ $(i=0, \ldots, d)$, $D = \{ T_0 = 0 \}$ and
$g = \log \max \{ a_0, a_1|z_1|, \ldots, a_d |z_d| \}$ for $a_0, a_1, \ldots, a_d \in \RR_{>0}$.
Then $g$ is a Green function of $D$, and
\[
\lambda_{\max}^{\mathrm{asy}}(D, g) = \hat{\mu}_{\mathrm{ess}}(D, g) = \log \max \{ a_0, \ldots, a_d \}.
\]
In particular, $(D, g)$ is big (resp. pseudo-effective) if and only if
$\max \{ a_0, \ldots, a_d \} > 1$ (resp. $\max \{ a_0, \ldots, a_d \} \geqslant 1$).

\medskip
Let us see the above facts.
The first assertion is obvious.
Furthermore, by 
Proposition~\ref{Pro:ess min g}, 
\[
\hat{\mu}_{\mathrm{ess}}(D, g) = g(\eta_0) = \log \max \{ a_0, \ldots, a_d \},
\] 
so that,
by 
Corollary~\ref{cor:finite:asy:max}, 
it is sufficient to show that
\[
\log \max \{ a_0, \ldots, a_d \} \leqslant \lambda_{\max}^{\mathrm{asy}}(D, g).
\]

We choose $a_{i_0}$ with $a_{i_0} = \max \{ a_0, \ldots, a_d \}$.
We set $w_i = T_i/T_{i_0}$ $(i=0, \ldots, d)$, $D' = D + (z_{i_0})$ and $g' = g - \log |z_{i_0}|$.
Note that $D' = \{ T_{i_0} = 0 \}$ and 
\[
g' = \log \max \{ a_0 |w_0|,\ldots, a_{i_0-1}|w_{i_0-1}|, a_{i_0}, a_{i_0+1}|w_{i_0+1}|, \ldots, a_d|w_d| \}.
\]
Moreover, by 
Proposition~\ref{prop:lambda:mu:rational:equiv}, 
$\lambda_{\max}^{\mathrm{asy}}(D, g) = \lambda_{\max}^{\mathrm{asy}}(D', g')$.
Therefore, we may assume that $a_0 = \max \{ a_0, \ldots, a_d \}$.

Let us see $\| 1 \|_g = 1/a_0$. Note that
\[
| 1 |_g = \frac{1}{\max \{ a_0, a_1|z_1|, \ldots, a_d|z_d| \}},
\]
so that, as $\max \{ a_0, a_1|z_1|, \ldots, a_d|z_d| \} \geqslant a_0$, we have $| 1 |_g \leqslant 1/a_0$ on 
$X^{\mathrm{an}}$.  Furthermore $| 1 |_g(\eta_0) = 1/a_0$, as desired.

The above observation shows 
$1 \in \mathcal{F}^{\log a_0}(H^0(D))$. Thus, $\log a_0 \leqslant \lambda_{\max}(D, g)$, that is, 
$\log a_0 \leqslant  \lambda_{\max}^{\mathrm{asy}}(D, g)$, as required.

\medskip
Let us consider a natural homomorphism
$\alpha : K^{d+1} \otimes \OO_{\PP^d_K} \to \OO_{\PP^d_K}(1)$ given by $\alpha(e_i) = T_i$ and
a norm $\|\ndot\|$ on $K^{d+1}$ given by 
\[
\| (x_0, \ldots, x_n) \| = \max \{ (1/a_0)|x_0|, \ldots, (1/a_d)|x_d| \}.
\]
Then the Green function $g$ is induced by the quotient metric of $\OO_{\PP^d_K}(1)$ by
$\alpha$ and $\|\ndot\|$.
\end{Example}

\subsection{Algebraic dynamic system over a trivially valued field}
Let $f : X \to X$ be a surjective endomorphism of an integral projective scheme over a trivially valued field $K$.
Let $D$ be an $\RR$-Cartier divisor on $X$ such that
$f^*(D) = dD + (\varphi)$ for some $d \in \RR_{>1}$ and $\varphi \in \Rat(X)^{\times}_{\RR}$. 
By Proposition~\ref{prop:unique:existence:canonical:Green:function},
we can see that there is a unique Green function $g$ of $D$ such that
$f^*(D, g) = d(D, g) + \widehat{(\varphi)}$.
Then we have the following:

\begin{Proposition}
\label{prop:height:dynamic:system}
\begin{enumerate}
\renewcommand{\labelenumi}{\rom{(\arabic{enumi})}}
\item
$h^{\mathrm{an}}_{(D, g)}(f^{\mathrm{an}}(\xi)) = d h^{\mathrm{an}}_{(D, g)}(\xi)$ for all $\xi \in X^{\mathrm{an}}$.

\item For $\xi \in X^{\mathrm{an}}$, if $(f^{\mathrm{an}})^n(\xi) =  (f^{\mathrm{an}})^m(\xi)$ for some
integers $n, m$ with $0 \leqslant n < m$,
then $h^{\mathrm{an}}_{(D, g)}(\xi) = 0$.

\item $h^{\mathrm{an}}_{(D, g)}(\eta_0) = 0$, that is, $g(\eta_0) = 0$.

\item
If $D$ is semiample, then $h^{\mathrm{an}}_{(D, g)}(\xi) \geqslant 0$ for all $\xi \in X^{\mathrm{an}}$.

\item
If $D +(s)$ is effective for some $s \in \Rat(X)^{\times}_{\RR}$,
then, for any $\epsilon > 0$, there is $\psi_{\epsilon} \in \Rat(X)^{\times}_{\RR}$ such that
$(D, g+\epsilon) + \widehat{(\psi_{\epsilon})}$ is effective.
\end{enumerate}
\end{Proposition}

\begin{proof}
(1) Indeed, by Proposition~\ref{prop:basic:prop:height},
\[
h^{\mathrm{an}}_{(D, g)}(f^{\mathrm{an}}(\xi)) = h^{\mathrm{an}}_{f^*(D, g)}(\xi) = h^{\mathrm{an}}_{d(D, g) + \widehat{(\varphi)}}(\xi)
= d h^{\mathrm{an}}_{(D, g)}(\xi).
\]

\smallskip
(2) By virtue of (1), \[
d^n h^{\mathrm{an}}_{(D, g)}(\xi) = h^{\mathrm{an}}_{(D, g)}((f^{\mathrm{an}})^n(\xi))
= h^{\mathrm{an}}_{(D, g)}((f^{\mathrm{an}})^m(\xi)) = d^m h^{\mathrm{an}}_{(D, g)}(\xi),
\]
and hence the assertion follows.

\smallskip
(3) is a consequence of (2) because $f^{\mathrm{an}}(\eta_0) = \eta_0$.

\smallskip
(4) By Proposition~\ref{prop:basic:prop:height}, there is a constant $C$ such that
$h^{\mathrm{an}}_{(D, g)}(\xi) \geqslant C$ for all $\xi \in X^{\mathrm{an}}$.
In particular, $h^{\mathrm{an}}_{(D, g)}((f^{\mathrm{an}})^n(\xi)) \geqslant C$, that is,
by (1), $h^{\mathrm{an}}_{(D, g)}(\xi) \geqslant C/d^n$ for all $n > 0$. Thus the assertion follows.

\smallskip
(5) follows from Proposition~\ref{prop:canonical:Dirichlet}.
\end{proof}

\section{the Dirichlet property over a trivially valued field}
In this section, we study the Dirichlet property in the setting of Arakelov geometry over a trivially valued field. We let $K$ be a field and $|\ndot|$ be the trivial absolute value on $K$.
Let $X$ be an integral projective scheme over $\Spec K$.

\begin{Definition}
Let $(D,g)$ be an adelic $\mathbb R$-Cartier divisor on $X$. We say that $(D,g)$ \emph{satisfies the Dirichlet property} if it is $\mathbb R$-linearly equivalent to an effective adelic $\mathbb R$-Cartier divisor (see Definition \ref{Def:principal} and Remark \ref{Rem:effective}).
\end{Definition}

\begin{Proposition}\label{Pro: pseudo-effectivity}
Let $(D,g)$ be an adelic $\mathbb R$-Cartier divisor on $X$. We assume that the $\mathbb R$-Cartier divisor $D$ is big. Then $(D,g)$ is pseudo-effective if and only if $(D,g+\varepsilon)$ satisfies the Dirichlet property for any $\varepsilon>0$.
\end{Proposition}
\begin{proof}
Suppose that $(D,g)$ is pseudo-effective. By Corollary \ref{Cor: pseudo-effectivity} one has \[\lambda_{\max}^{\mathrm{asy}}(D,g)\geqslant 0.\]  Then for $\varepsilon>0$ one has \[\lambda_{\max}^{\mathrm{asy}}(D,g+\varepsilon)=\lambda_{\max}^{\mathrm{asy}}(D,g)+\varepsilon>0.\]
Therefore, by Proposition \ref{Pro: bigness and lambda}, we obtain that $(D,g+\varepsilon)$ is big. Therefore $(D,g+\varepsilon)$ satisfies the Dirichlet property.

Conversely, if $(D,g+\varepsilon)$ satisfies the Dirichlet property, then $(D,g+\varepsilon)$ is pseudo-effective, and hence by Corollary \ref{Cor: pseudo-effectivity} we obtain that $\lambda_{\max}^{\mathrm{asy}}(D,g+\varepsilon)=\lambda_{\max}^{\mathrm{asy}}(D,g)+\varepsilon\geqslant 0$. Therefore, if $(D,g+\varepsilon)$ satisfies the Dirichlet property for any $\varepsilon>0$, then one has $\lambda_{\max}^{\mathrm{asy}}(D,g)\geqslant 0$. By Corollary \ref{Cor: pseudo-effectivity} we obtain that $(D,g)$ is pseudo-effective.
\end{proof}

\begin{Definition}
We say that {\em the rank of $\Pic(X)$ is one} if $\dim_{\QQ} \Pic_{\QQ}(X) = 1$ (or equivalently
$\dim_{\RR} \Pic_{\RR}(X)  = 1$). In other words,
there is a $\KK$-Cartier divisor $A$ on $X$ such that, for any $\KK$-Cartier divisor $D$ on $X$,
we can find $a \in \KK$ and $\varphi \in \Rat(X)^{\times}_{\KK}$ with
$D = a A + (\varphi)$, where $\KK$ is either $\QQ$ or $\RR$. 
Obviously, $A$ can be taken as an ample Cartier divisor.
\end{Definition}

\begin{Proposition}
\label{prop:equiv:Pic:one}
We assume that $\dim X = 1$. Then the following are equivalent:
\begin{enumerate}
\renewcommand{\labelenumi}{(\arabic{enumi})}
\item
The rank of $\Pic(X)$ is one.

\item
Every element of $\Pic^0(X)$ is of finite order.
\end{enumerate}
\end{Proposition}

\begin{proof}
We fix an ample Cartier divisor $A$ on $X$.

(1) $\Longrightarrow$ (2): For a Cartier divisor $D$ on $X$ with $\deg(D) = 0$,
there are $a \in \QQ$ and $\varphi \in \Rat(X)^{\times}_{\QQ}$ such that
$D=aA + (\varphi)$. Note that $a= 0$ because $\deg(D) = 0$ and $\deg(A) > 0$.
We choose a positive integer $n$ such that
$\varphi^n  \in \Rat(X)^{\times}$, so that $nD \in \PDiv(X)$, as required.

(2) $\Longrightarrow$ (1): Let $D$ be a $\QQ$-Cartier divisor on $X$.
If we set $H:= \deg(A)D - \deg(D) A$, then $\deg(H) = 0$.
Let $n_0$ be a positive integer such that $n_0D$ is a Cartier divisor.
Then, as $\OO_X(n_0H)$ yields an element of $\Pic^0(X)$, we can find a positive integer
$n_1$ such that $n_1n_0H = (f)$ for some $f \in \Rat(X)^{\times}$, so that
\[
D = \frac{\deg(D)}{\deg(A)} A + \frac{1}{n_1n_0\deg(A)}(f),
\]
as desired.
\end{proof}

\begin{Remark}
If either $X = \PP^1_K$ or $K$ is the algebraic closure of a finite field,
then the above condition on $\Pic^0(X)$ is satisfied.
\end{Remark}

\subsection{Adelic $\mathbb R$-Cartier divisors on curves}

In the rest of the section, we assume that $X$ is a regular projective curve over $\Spec K$ such that the rank of $\Pic(X)$ is one. 
Under the above definition, for any $\mathbb R$-Cartier divisor $D$ on $X$ with
$\deg(D)=0$, there exists an element $f\in\mathrm{Rat}(X)^{\times}_{\mathbb R}$ such that $D=(f)$. 
Indeed, if $A$ is an ample Cartier divisor on $X$, then there are $a \in \RR$ and $f \in \Rat(X)^{\times}_{\RR}$
such that $D = a A + (f)$. Since $\deg(D) = 0$, we have $a = 0$, as required.

Recall that the Berkovich space associated with $X$ can be illustrated by an infinite tree $X^{\mathrm{an}}$, where the root vertex $\eta_0$ corresponds to the generic point $\eta$ of $X$ with the trivial absolute value on the field $\mathrm{Rat}(X)$ of rational functions. The leaves are indexed by closed points of $X$.

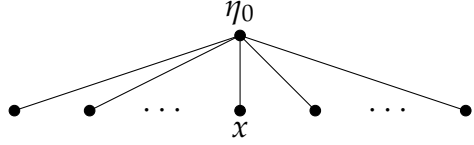
\begin{center}
\begin{tikzpicture}
\filldraw(0,1) circle (2pt) node[align=center, above]{$\eta_0$};
\filldraw(-3,0) circle (2pt) ;
\draw (-1,0) node{$\cdots$};
\filldraw(-2,0) circle (2pt) ;
\filldraw(-0,0) circle (2pt) node[align=center, below]{$x$} ;
\filldraw(1,0) circle (2pt) ;
\draw (2,0) node{$\cdots$};
\filldraw(3,0) circle (2pt) ;
\draw (0,1) -- (0,0);
\draw (0,1) -- (-3,0);
\draw (0,1) -- (1,0);
\draw (0,1) -- (-2,0);
\draw (0,1) -- (3,0);
\end{tikzpicture}
\end{center}

Let $x$ be a closed point of $X$. We parametrise the branch linking $\eta_0$ and $x$ by $t\in[0,+\infty]$, where $t=0$ correspond to the point $\eta_0$; the point $t=+\infty$ correspond to the point $x$ with the trivial absolute value on the residue field $\kappa(x)$, and any  $t\in \,]0,+\infty[$ corresponds to the generic point $\eta$ with the following absolute value on $\mathrm{Rat}(X)$:  \[|\ndot|_{x,t}=\mathrm{e}^{-t\,\mathrm{ord}_x(\cdot)},\]
where $\mathrm{ord}_x(\cdot)$ is the discrete valuation on $\mathrm{Rat}(X)$ corresponding to $x$. 

The topology on each branch identifies with the usual topology on $[0,+\infty]$ by this parametrisation and hence each branch is compact. However, any open neighbourhood of $\eta_0$ in $X^{\mathrm{an}}$ contains all but a finite number of branches. 
Namely, a subset $U$ of $X^{\mathrm{an}}$ is open if and only if
the following conditions are satisfied:
\begin{enumerate}
\renewcommand{\labelenumi}{(\roman{enumi})}
\item
$U \cap [\eta_0, x]$ is open for all closed points $x$ of $X$.

\item
If $\eta_0 \in U$, then $[\eta_0, x] \subseteq U$ for all but a finitely many closed points $x$ of $X$.
\end{enumerate}
Note that $X^{\mathrm{an}}$ is compact with this topology.

If $s$ is a non-zero rational function on $X$, on the interval $[\eta_0,x]$ one has \begin{equation}\label{Equ:1}-\log|s|(\xi)=t(\xi)\ord_x(s)\in[-\infty,+\infty],\quad \xi\in[\eta_0,x].\end{equation}
This function is linear on each branch $[\eta_0,x[$ with respect to the parametrisation $t$.

\subsection{Numerical criteria of pseudo-effectivity and Dirichlet property}
In this subsection, we consider a numerical criterion of pseudo-effectivity and the Dirichlet property.

\begin{Definition}
Let $(D,g)$ be an adelic $\mathbb R$-Cartier divisor on $X$. We consider $g$ as a continuous map from $X^{\mathrm{an}}$ to $[-\infty,+\infty]$. For any closed point $x$ of $X$, we define
\[\mu_x(g)=\inf_{\xi\in\,]\eta_0,x[}\frac{g(\xi)}{t(\xi)}\in \mathbb R.\]
Clearly $\mu_x(g)\geqslant 0$ if and only if the function $g$ is bounded from below by $0$ on $]\eta_0,x[$.
\end{Definition}

\begin{Proposition}
Let $(D,g)$ be an adelic $\mathbb R$-Cartier divisor on $X$. For all but finitely many closed point $x$ in $X$, one has $\mu_x(g)\leqslant 0$.
\end{Proposition}
\begin{proof}
Since $g$ is a Green function, it extends to a continuous real-valued function on some $U^{\mathrm{an}}$, where $U$ is a non-empty Zariski open subset of $X$. Hence for any closed point $x$ in $U$, then function $g$ is bounded on $[\eta_0,x]$. Therefore $\mu_x(g)\leqslant 0$ for any closed point $x\in U$. Since there are only a finite number of closed points outside of $U$, the proposition is proved.
\end{proof}

The above result permits to define an invariant $\mu_{\mathrm{tot}}(g)$ as follows
\begin{equation}\label{Equ:mutotal}\mu_{\mathrm{tot}}(g):=\sum_{x\in X^{(1)}}\mu_x(g)[\kappa(x):K]\in[-\infty,+\infty[\,,\end{equation}
where $X^{(1)}$ denotes the set of closed points of $X$, considered as a discrete measure space such that each point $x\in X^{(1)}$ has mass $1$, and the summation means the integration on this measure space. In the case where the set of $x\in X^{(1)}$ such that $\mu_x(g)<0$ is uncountable, one has $\mu_{\mathrm{tot}}(g)=-\infty$. Otherwise the set $\{x\in X^{(1)}\,:\,\mu_x(g)\neq 0\}$ is countable. If it is infinite then we can write it as a sequence $\{x_n\}_{n\in\mathbb N}$ and one has
\[\mu_{\mathrm{tot}}(g)=\sum_{n\in\mathbb N}\mu_{x_n}(g)[\kappa(x_n):K].\]
The sum does not depend on the choice of the sequence since $\mu_{x_n}(g)<0$ for all but finitely many $n\in\mathbb N$.

\begin{Lemma}\label{Lem:4}
Let $X$ be a regular projective curve over $\Spec K$ and $(D,g)$ be an adelic $\mathbb R$-Cartier divisor on $X$. 
\begin{enumerate}
\renewcommand{\labelenumi}{\rom{(\arabic{enumi})}}
\item For any non-zero element $s\in \mathrm{Rat}(X)_{\mathbb R}^{\times}$, one has
\[\mu_x(g-\log|s|)=\mu_x(g)+\mathrm{ord}_x(s).\]
\item One has $\mu_x(g)\leqslant\mathrm{ord}_x(D)$.
\end{enumerate}
\end{Lemma}
\begin{proof}
(1) For any $s\in \mathrm{Rat}(X)^{\times}_{\mathbb R}$ one has 
\[\mu_x(g-\log|s|)=\inf_{\xi\in\,]\eta_0,x[}\frac{g(\xi)-\log|s|(\xi)}{t(\xi)}=\mu_x(g)+\mathrm{ord}_x(s),\]
where the second equality comes from \eqref{Equ:1}.

(2) We let $s$ be an element in $\mathrm{Rat}(X)_{\mathbb R}^{\times}$ which defines $D$ locally on a Zariski open neighbourhood of $x$. Then the function $g+\log|s|$ extends continuously to $[\eta_0,x]$ and hence is bounded. Therefore $\mu_x(g+\log|s|)\leqslant 0$. By (1), we obtain that $\mu_x(g)\leqslant\mathrm{ord}_x(D)$.
\end{proof}

\begin{Theorem}\label{Pro:Dirichlet}
Let $X$ be a regular projective curve over $\Spec K$ such that
the rank of $\Pic(X)$ is one
and $(D,g)$ be an adelic $\mathbb R$-Cartier divisor on $X$. The following conditions are equivalent.
\begin{enumerate}
\renewcommand{\labelenumi}{\rom{(\arabic{enumi})}}
\item The adelic $\mathbb R$-Cartier divisor $(D,g)$ satisfies the Dirichlet property.
\item For all but a finite number of closed points $x\in X$, one has $\mu_x(g)\geqslant 0$, and $\mu_{\mathrm{tot}}(g)\geqslant 0$.
\end{enumerate}
\end{Theorem}
\begin{proof}
``(1)$\Longrightarrow$(2)'': We assume that there exists $s\in\mathrm{Rat}(X)^{\times}_{\mathbb R}$ such that $(D,g)+\widehat{(s)}$ is effective. By Lemma \ref{Lem:4} (1) one has
\begin{equation}\label{Equ: mu g-logs}\mu_x(g-\log|s|)=\mu_x(g)+\mathrm{ord}_x(s)\end{equation}
for any $x\in X$. Therefore, for all but a finite number of
$x\in X^{(1)}$,  
one has
\[\mu_x(g)=\mu_x(g-\log|s|)\geqslant 0.\]
Moreover, by \eqref{Equ: mu g-logs}
\[\mu_{\mathrm{tot}}(g)=\mu_{\mathrm{tot}}(g-\log|s|)-\sum_{x\in X^{(1)}}\mathrm{ord}_x(s)[\kappa(x):K]=\mu_{\mathrm{tot}}(g-\log|s|)\geqslant 0.\]
Therefore the condition (2) holds.

``(2)$\Longrightarrow$(1)'': Suppose that (2) is true. Since $g$ is a Green function of $D$, for any $x\in X^{(1)}$ which does not belong to the support of $D$, the function $g$ extends to a continuous function on a open subset of $X^{\mathrm{an}}$ containing $[\eta_0,x]$ and hence is bounded on this compact set. In particular, one has $\mu_x(g)\leqslant 0$ for such point $x$. Hence by the hypothesis in the condition (2), we obtain that $\mu_x(g)=0$ for all but a finite number of closed points $x_1,\ldots,x_n$ in $X$.
Let $a_1,\ldots,a_n$ be real numbers such that $ a_i\leqslant\mu_{x_i}(g)$ and
\[\sum_{i=1}^na_i[\kappa(x_i):K]=0.\]
This is possible since 
\[\sum_{i=1}^n\mu_{x_i}(g)[\kappa(x_i):K]=\mu_{\mathrm{tot}}(g)\geqslant 0.\]
Let
\[D'=\sum_{i=1}^na_i[x_i].\]
Since $\deg(D')=0$, it is an principal $\mathbb R$-Cartier divisor. Let $s\in \mathrm{Rat}(X)^{\times}_{\mathbb R}$ such that $-D'=(s)$. By Lemma \ref{Lem:4} (2) one has $\mu_x(g)\leqslant\mathrm{ord}_x(D)$ for any $x$. Hence $D- D'$ is effective. Moreover, Lemma \ref{Lem:4} (1) shows that 
\[\mu_x(g-\log|s|)=
\mu_x(g)+\mathrm{ord}_x(s)=\begin{cases}
\mu_{x_i}(g)-a_i\geqslant 0,&\text{if $x=x_i$},\\
\mu_{x}(g)\geqslant 0,&\text{else}.
\end{cases}\]
Therefore $g-\log|s|$ is non-negative.
\end{proof}

\begin{Theorem}\label{Pro: Dirichlet property and mu tot}
Let $X$ be a regular projective curve over $\Spec K$ such that 
the rank of $\Pic(X)$ is one
and $(D,g)$ be an adelic $\mathbb R$-Cartier divisor on $X$. The following conditions are equivalent.
\begin{enumerate}
\renewcommand{\labelenumi}{\rom{(\arabic{enumi})}}
\item For any $\varepsilon>0$, the Dirichlet property holds for $(D,g+\varepsilon)$.
\item One has $\mu_{\mathrm{tot}}(g)\geqslant 0$. 
\end{enumerate}
\end{Theorem}
\begin{proof}
By Theorem \ref{Pro:Dirichlet}, if (1) holds, then for any $\varepsilon>0$ one has
$\mu_{\mathrm{tot}}(g+\varepsilon)\geqslant 0$.
Since the function $\varepsilon\mapsto \mu_y(g+\varepsilon)$ is decreasing and converges to $\mu_y(g)$, by the monotone convergence theorem we obtain $\mu_{\mathrm{tot}}(g)\geqslant 0$.

Conversely, assume that the inequality $\mu_{\mathrm{tot}}(g)\geqslant 0$ holds.
We claim that $g(\eta_0)\geqslant 0$. Otherwise one has $\mu_x(g)=-\infty$ for any $x\in X^{\mathrm{an}}$ and hence $\mu_{\mathrm{tot}}(g)=-\infty$. Let $\varepsilon$ be a positive number. The set $\{y\in X^{\mathrm{an}}\,:\,g(y)+\varepsilon>0\}$ is an open subset of $X^{\mathrm{an}}$ containing $\eta_0$. Therefore, there exists a 
finite  
subset $S$ of $X^{(1)}$ such that $g(y)+\varepsilon\geqslant 0$ for any $x\in X^{(1)}\setminus S$ and any $y\in [\eta_0,x]$. In particular, for any $x\in  X^{(1)}\setminus S$, one has 
$\mu_x(g+\varepsilon)\geqslant 0$. 
Moreover, one has $\mu_{\mathrm{tot}}(g+\varepsilon)\geqslant\mu_{\mathrm{tot}}(g)\geqslant 0$. By Theorem \ref{Pro:Dirichlet}, we obtain that the Dirichlet property holds for $(D,g+\varepsilon)$.
\end{proof}

\begin{Remark}\label{Rem: obstruction}
Let $X$ be a regular projective curve over $\Spec K$ such that 
the rank of $\Pic(X)$ is one.
Let $(D,g)$ be an adelic $\mathbb R$-Cartier divisor on $X$. We assume that $D$ is big. By Proposition \ref{Pro: pseudo-effectivity}, the conditions in Theorem \ref{Pro: Dirichlet property and mu tot} are equivalent to the pseudo-effectivity of $(D,g)$. Note that the function $\lambda_{\max}^{\mathrm{asy}}(\ndot)$ is derivable at $(D,g)$ along any direction in  $C^0(X^{\mathrm{an}})$. Moreover, the directional derivatives of $\lambda_{\max}^{\mathrm{asy}}(\ndot)$ form a positive linear functional on $C^0(X^{\mathrm{an}})$, which identifies with the Dirac measure on $\eta_0$. Therefore the above results suggest that the functional obstruction to the Dirichlet property proposed in \cite{DsysDirichlet} may not be the only obstruction.
\end{Remark}

\begin{Example}
\label{exam:projective:line:Dirichlet}
We assume $X = \PP^1_K = \Proj(K[T_0, T_1])$. We set $\pmb{0} = (1:0)$, $\pmb{\infty} = (0:1)$,
$z = T_1/T_0$, $D = \{ T_0 = 0 \}$ and
$g = \log \max \{ 1, |z| \}$.
Then $g$ is a Green function of $D$ and $(D, g)$ is pseudo-effective by 
Example~\ref{exam:projective:space:big}. 
Moreover, since
\[
g(\xi) = \begin{cases}
t(\xi) & \text{if $x = \pmb{\infty}$}, \\
0 & \text{otherwise}
\end{cases}
\]
for a closed 
point $x \in X$ and  $\xi \in [\eta_0, x]$, we have
\[
\mu_x(g) = \begin{cases}
1 & \text{if $x = \pmb{\infty}$},\\
0 & \text{otherwise}.
\end{cases}
\]
We choose distinct countably many closed points
$x_1, \ldots, x_n, \ldots$ in $\PP^1_K \setminus \{ \pmb{0}, \pmb{\infty} \}$.
Here we consider the following continuous function $\psi$ on $\PP^{1, \mathrm{an}}_K$:
for a closed point $x$ of $\PP^1_K$ and $\xi \in [\eta_0, x]$,
\[
\psi(\xi) := \begin{cases}
{\displaystyle \theta_{\frac{1}{2^n [\kappa(x_n) : K]}}(t(\xi))} & \text{if $x = x_n$ for some $n$}, \\[1ex]
0 & \text{if $x \not\in \{ x_1, \ldots, x_n, \ldots \}$},
\end{cases}
\]
where $\theta_a$ ($a \in \RR_{>0}$) is a continuous function on $[0, \infty]$ given by
\[
\theta_a(t) := \begin{cases}
a t & t \in [0, 1], \\
a & t \in [1, \infty].
\end{cases}
\]
We introduce a new Green function $g'$ given by $g' := g - \psi$.
Then
\[
\mu_x(g') = \begin{cases}
{\displaystyle - \frac{1}{2^n [\kappa(x_n) : K]}} & \text{if $x = x_n$ for some $n$}, \\[2ex]
\mu_x(g) & \text{if $x \not\in \{ x_1, \ldots, x_n, \ldots \}$}.
\end{cases}
\]
Thus,
\[
\mu_{\mathrm{tot}}(g') = 1 - \sum_{n=1}^{\infty} (1/2)^n = 0.
\]
Therefore $(D, g')$ is pseudo-effective, but $(D, g')$ has no Dirichlet's property.
\end{Example}

\begin{Corollary}
\label{cor:Dirichlet:endomorphism:curve}
Let $f : X \to X$ be a surjective endomorphism of $X$ over $K$ and $D$ be an $\RR$-Cartier divisor on $X$
such that $f^*(D) = d D + (\varphi)$ for some $\varphi \in \Rat(X)^{\times}_{\RR}$ and
$d \in \RR_{>1}$. Let $(D, g)$ be the canonical compactification of $D$ with respect to $f$.
We assume that $\dim X = 1$, the rank of $\Pic(X)$ is one and
$\deg(D) \geqslant 0$.
Then $(D, g)$ satisfies the Dirichlet property.
\end{Corollary}

\begin{proof}
We set $D = a_1 D_1 + \cdots + a_r D_r$ for some Cartier divisors
$D_1, \ldots, D_r$ on $X$ and $a_1, \ldots, a_r \in \RR$.
Let $A$ be an ample and effective Cartier divisor on $X$.
Since the rank of $\Pic(X)$ is one,
for each $i$,
there are $b_i \in \QQ$ and $s_i \in \Rat(X)^{\times}_{\QQ}$ with $D_i = b_i A + (s_i)$.
Thus $D = aA + (s)$ for some $a \in \RR$ and $s \in \Rat(X)^{\times}_{\RR}$.
Note that $a \geqslant 0$ because $\deg(D) \geqslant 0$.
Therefore, $D$ is semiample and $D + (s^{-1})$ is effective.

We set $U = X \setminus \Supp(D)$.
A local equation of $D$ over $U$ is given by $1$.
Therefore, by (4) in Proposition~\ref{prop:height:dynamic:system}, for all $\xi \in U^{\mathrm{an}}$,
$g(\xi) = h_{(D,g)}^{\mathrm{an}}(\xi) \geqslant 0$. Therefore, $g \geqslant 0$ on $[\eta_0, x]$
for all but finitely many closed points $x$, and hence
$\mu_x(g) \geqslant 0$ for all but finitely many closed points $x$.

Moreover, by Theorem~\ref{Pro: Dirichlet property and mu tot}  together with
(5) in Proposition~\ref{prop:height:dynamic:system},
$\mu_{\mathrm{tot}}(g) \geqslant 0$. 
Thus $(D, g)$ satisfies the Dirichlet property
by Theorem~\ref{Pro:Dirichlet}.
\end{proof}

\subsection{The plurisubharmonic case}
Here we apply the criterion of the previous subsection to a plurisubharmonic Green function.

\begin{Proposition}
Let $X\rightarrow\Spec K$ be a regular projective curve on $\Spec K$, $D$ be an $\mathbb R$-Cartier divisor on $X$ and $g$ be a plurisubharmonic Green function of $D$. For any closed point $x$ of $X$, the restriction of the function $g$ on $[\eta_0,x[$ is concave, where we consider the parametrisation $t:[\eta_0,x[\,\rightarrow[0,+\infty[$.
\end{Proposition}
\begin{proof}
Since uniform limits and positive linear combinations of concave functions is still a concave function, it suffices to treat the case where $D$ is a Cartier divisor and the metric $\phi_g$ on $\mathcal O_X(D)$ corresponding to the Green function $g$ is a quotient metric. Let $(E,\|\ndot\|)$ be a finite dimensional normed vector space over $K$ and $f:X\rightarrow\mathbb P(E)$ be a $K$-morphism such that $f^*(\mathcal O_E(1))\cong\mathcal O_X(D)$ and that $\phi_g$ is the quotient metric induced by $((E,\|\ndot\|),f)$. Let $x$ be a closed point of $X$ and $s$ be a section of $\mathcal O_X(D)$ over a Zariski open neighbourhood $U$ of $X$, which trivialises the invertible sheaf $\mathcal O_X(D)$ on $U$. Let $(\alpha_i)_{i=1}^n$ be an orthogonal basis of $(E^\vee,\|\ndot\|^\vee)$, where $\|\ndot\|^\vee$ denotes the dual norm of $\|\ndot\|$. Recall that for any $(\lambda_1,\ldots,\lambda_n)\in K^n$ one has
\[\|\lambda_1\alpha_1+\cdots+\lambda_n\alpha_n\|=\max_{i\in\{1,\ldots,n\}}|\lambda_i|\cdot\|\alpha_i\|^\vee.\]
We refer the readers to \cite[\S1.3]{Extension} for the existence of an orthogonal basis of $(E^\vee,\|\ndot\|^\vee)$. We write the dual section $s^\vee\in H^0(U,\mathcal O_X(D)^\vee)$ as a linear combination
\[s^\vee=u_1\alpha_1+\cdots+u_n\alpha_n,\]
where $u_1,\ldots, u_n$ are regular functions on $U$. One has
\[-\log |s|=\log |s^\vee|=\max_{i\in\{1,\ldots,n\}}(\log|u_i|+\log\|\alpha_i\|^\vee).\]
Since each function $\log|u_i|$ is linear on $[\eta_0,x[$ with respect to the parametrisation $t:[\eta_0,x[\,\rightarrow[0,+\infty[$, we obtain that the function $-\log |s|$ is concave on $[\eta_0,x[$. Since the functions $g$ and $-\log |s|$ differ by a linear function on $[\eta_0,x[$, the function $g$ is also concave.
\end{proof}

\begin{Corollary}
\label{cor:plurisubharmonic:pseudoeffective:Dirichlet}
Let $X$ be a regular projective curve over $\Spec K$ such that the rank of $\Pic(X)$ is one and $(D,g)$ be an adelic $\mathbb R$-Cartier divisor. We assume that the Green function $g$ is plurisubharmonic. Then the adelic $\mathbb R$-Cartier divisor $(D,g)$ satisfies the Dirichlet property if and only if it is pseudo-effective.
\end{Corollary}
\begin{proof}
It is clear that any adelic $\mathbb R$-Cartier divisor satisfying the Dirichlet property is pseudo-effective. 
Let $(D,g)$ be a pseudo-effective adelic $\mathbb R$-Cartier divisor. We claim that $g(\eta_0)\geqslant 0$. In fact, let $(D_1,g_1)$ be a big adelic $\mathbb R$-Cartier divisor. For any $\varepsilon>0$, the adelic $\mathbb R$-Cartier divisor $(D+\varepsilon D_1,g+\varepsilon g_1)$ is big. Hence by Corollary \ref{cor:finite:asy:max}
and Proposition \ref{Pro: bigness and lambda}, one has
\[g(\eta_0)+\varepsilon g_1(\eta_0)=\widehat{\mu}_{\mathrm{ess}}(D+\varepsilon D_1,g+\varepsilon g_1)>0,\]
where the equality comes form Proposition \ref{Pro:ess min g}. Since $\varepsilon>0$ is arbitrary, we obtain $g(\eta_0)\geqslant 0$. 

Let $x$ be a closed point of $X$.
We choose a Zariski open set $U$ containg $x$ such that
a local equation of $D$ over $U$ is given by $f$.
Then there is a continuous function $u$ on $U^{\mathrm{an}}$ such that
$g = u - \log |f|$ over $U^{\mathrm{an}}$. Note that $[\eta_0, x] \subseteq U^{\mathrm{an}}$ and
\[
\ord_{x}(f) = \lim_{t(\xi)\rightarrow+\infty} \frac{-\log |f|(\xi)}{t(\xi)} =
\lim_{t(\xi)\rightarrow+\infty} \frac{g(\xi)}{t(\xi)},
\]
so that
\[
\mathrm{ord}_x(D)=\lim_{t(\xi)\rightarrow+\infty}\frac{g(\xi)}{t(\xi)}=\lim_{t(\xi)\rightarrow+\infty}\frac{g(\xi)-g(\eta_0)}{t(\xi)}.
\]
Since the restriction of the function $g$ on $[\eta_0,x[$ is concave,
the function
\[\xi\longmapsto\frac{g(\xi)-g(\eta_0)}{t(\xi)}\]
is decreasing with respect to the parametrisation $t(\xi)$.  
We then deduce that, for any $\xi\in\,]\eta_0,x[$ one has
\[\frac{g(\xi)}{t(\xi)}\geqslant\frac{g(\xi)-g(\eta_0)}{t(\xi)}\geqslant\mathrm{ord}_x(D),\]which implies that
\[\mu_x(g)=\lim_{t(\xi)\rightarrow +\infty}\frac{g(\xi)}{t(\xi)}=\mathrm{ord}_x(D).\]
Hence $\mu_x(g)=0$ for all but finitely many closed point $x$ in $X$. Finally, since $\overline D$ is pseudo-effective, the $\mathbb R$-Cartier divisor $D$ is pseudo-effective, and hence 
\[\mu_{\mathrm{tot}}(g)=\sum_{x\in X^{(1)}}\mathrm{ord}_x(D)\geqslant 0.\]
By Theorem \ref{Pro:Dirichlet}, we obtain that $(D,g)$ satisfies the Dirichlet property.
\end{proof}

\begin{Remark}
By the above corollary together with Proposition~\ref{prop:semiample:plurisubharmonic:endomorphism},
we can give an alternative proof of Corollary~\ref{cor:Dirichlet:endomorphism:curve}.
\end{Remark}

\bibliography{SuffDirichlet}
\bibliographystyle{plain}

\end{document}